\documentclass[11pt]{amsart}

\usepackage{geometry}
\usepackage{todonotes}

\usepackage{amsfonts, amssymb, amscd}
\numberwithin{equation}{section}

\usepackage[symbol]{footmisc}

\usepackage{bm}
\usepackage{verbatim}
\usepackage{mathrsfs}
\usepackage{graphicx}
\usepackage{tikz-cd}
\usepackage{subcaption}
\usepackage{listings}
\usepackage{subfiles}
\usepackage[toc,page]{appendix}
\usepackage{mathtools}
\usepackage{comment}
\usepackage{enumerate}
\usepackage{enumitem}
\usepackage[all]{xy}

\usepackage{graphicx}
\graphicspath{{images/}}

\usepackage{appendix}
\usepackage{hyperref}
\hypersetup{
	colorlinks=true,
	citecolor=red,
	linkcolor=blue,
	filecolor=magenta,      
	urlcolor=red,
}
\newcommand{\pp}{\mathcal{P}}
\newcommand{\Cb}{\mathcal{B}}
\newcommand{\Ivol}{\mathrm{Ivol}}

\newcommand{\Cz}{\mathcal{Z}}
\newcommand{\M}{\bm{M}}

\newcommand{\Cd}{\mathcal{D}}
\newcommand{\Cf}{\mathcal{F}}
\newcommand{\Oo}{\mathcal{O}}
\newcommand{\Ii}{\mathcal{I}}

\newcommand{\Cw}{\mathcal{W}}
\newcommand{\Cx}{\mathcal{X}}
\newcommand{\Cs}{\mathcal{S}}
\newcommand{\Cy}{\mathcal{Y}}
\newcommand{\Ch}{\mathcal{H}}
\newcommand{\Cr}{\mathcal{R}}
\newcommand{\Cu}{\mathcal{U}}
\newcommand{\Supp}{\mathrm{Supp}}
\newcommand{\Cm}{\mathcal{M}}

\theoremstyle{plain} 
\newtheorem{thm}{Theorem}[section] 
\newtheorem{lemma}[thm]{Lemma}
\newtheorem{prop}[thm]{Proposition}

\theoremstyle{definition} 
\newtheorem{defn}[thm]{Definition} 
\newtheorem{conj}[thm]{Conjecture}
\theoremstyle{remark} 
\newtheorem{rem}[thm]{Remark}

\begin{document}
	
	\title{On the boundedness of canonical models}
	\author{Junpeng Jiao}
	\email{jiao$\_$jp@tsinghua.edu.cn}
	\address{Yau Mathematical Sciences Center, Tsinghua University, Beijing, China}
	\keywords{canonical models, Iitaka fibration, polarised Calabi-Yau pairs, boundedness}
	
	\thanks{The author was partially supported by NSF research grant no: DMS-1952522 and by a grant from the Simons Foundation; Award Number: 256202.}

	\begin{abstract}
		It is conjectured that the canonical models of varieties (not of general type) are bounded when the Iitaka volume is fixed. We confirm this conjecture when the general fiber of the corresponding Iitaka fibration is in a fixed bounded family of polarized Calabi-Yau pairs. 
	\end{abstract}
	\maketitle
	
	\section{Introduction}
	Throughout this paper, we work over the complex numbers field $\mathbb{C}$.
	
	By analogy with the definition of volumes of divisors, the Iitaka volume of a $\mathbb{Q}$-divisor is defined as follows:
	Let $X$ be a normal projective variety and $D$ be a $\mathbb{Q}$-Cartier divisor. When the Iitaka dimension $\kappa(D)$ of $D$ is non-negative, the Iitaka volume of $D$ is defined to be
	\begin{equation*}
		\Ivol(D):=\limsup_{m\rightarrow \infty}\frac{\kappa(D)! h^0(X,\Oo_X(\lfloor mD\rfloor))}{m^{\kappa(D)}}
	\end{equation*}
	\begin{thm}\label{fibers in a good moduli space}
		Fix $\mathscr{C}$ a log bounded class of polarized log Calabi-Yau pairs, $\Ii\subset [0,1]\cap \mathbb{Q}$ a DCC set of rational numbers, $n\in \mathbb{Z}_{>0}$ and $v\in \mathbb{Q}_{>0}$. Suppose $(X,\Delta)$ is a klt pair of dimension $n$, $L$ is a divisor on $X$, and $f: X\rightarrow Z$ is an algebraic contraction which is birationally equivalent to the Iitaka fibration of $K_X+\Delta$.
		
		If the general fiber $(X_g,\Delta_g,L_g)$ of $f$ is in $\mathscr{C}$ and $\mathrm{coeff}\Delta\subset \Ii$, then 
		\begin{enumerate}
			\item $\Ivol(X,K_X+\Delta)$ is in a DCC set, and
			
			\item If $\Ivol(X,K_X+\Delta)=v$ is a constant, then $$\mathrm{Proj}\oplus _{m=0}^\infty H^0(X,\mathcal{O}_X(mK_X+\lfloor m\Delta\rfloor))$$
			is in a bounded family.
		\end{enumerate}
	\end{thm}
	
	Theorem (\ref{fibers in a good moduli space}) is a special case of the following conjecture.
	
	\begin{conj}\label{conjecture on boundedness}
		Let $n $ be a positive integer, $v$ a non-negative rational number, and $\Ii\subset [0,1]\cap \mathbb{Q}$ a DCC set of rational numbers. Let $\Cd(n,v,\Ii)$ be the set of varieties $Z$ such that:
		\begin{itemize}
			\item $(X,\Delta)$ is a klt pair of dimension $n$.
			\item $\mathrm{coeff} \Delta \subset \Ii$.
			\item $\Ivol(K_X+\Delta)=v$ is a constant.
			\item $f: X\dashrightarrow Z$ is the Iitaka vibration associated with $K_X+\Delta$, where
			\begin{equation*}
				Z=\mathrm{Proj}\oplus _{m=0}^\infty H^0(X,\mathcal{O}_X(m(K_X+\Delta))).
			\end{equation*}
		\end{itemize}
		Then $\Cd(n,v,\Ii)$ is in a bounded family.
	\end{conj}
	Note that by \cite{BCHM10}, the canonical ring $R(X,K_X+\Delta):=\oplus _{m=0}^\infty H^0(X,\mathcal{O}_X(m(K_X+\Delta)))$ is finitely generated, this means $Z=\mathrm{Proj}\oplus _{m=0}^\infty H^0(X,\mathcal{O}_X(m(K_X+\Delta)))$ is well defined and $v=\Ivol(X,K_X+\Delta)$ is a positive rational number. When $K_X+\Delta$ is big, Conjecture \ref{conjecture on boundedness} is proved by \cite{HMX14}, and when the general fiber of $f$ is $\epsilon$-lc Fano type, it is proved by \cite{Li20}.
	Almost at the same time, \cite{Bir21b} proved Conjecture \ref{conjecture on boundedness} when the general fiber of $f$ is in a bounded family.
	
	It is shown in \cite{HMX13} that the boundedness of varieties is connected with the DCC of volumes. We think that the following conjecture is closely related to Conjecture \ref{conjecture on boundedness}.
	\begin{conj}\label{conjecture on volume}
		Let $n\in\mathbb{N}$ and a DCC set $\Ii\subset[0,1]\cap \mathbb{Q}$. Then the set of Iitaka volumes
		$$\{\Ivol(K_X+\Delta)\ |\ \text{$(X,\Delta)$ is klt, $\mathrm{dim}X=n$ and $\mathrm{coeff}(\Delta)\subset \Ii$.} \}$$
		is a DCC set.
	\end{conj}
	The main idea is to prove the DCC of Iitaka volumes and the boundedness of the canonical models when the locus of singular fibers of the Iitaka fibrations is "bounded". We show that in this case, we can choose a uniform base such that the moduli part (see Theorem \ref{canonical bundle formula}) descends.
	
	To be precise, we are interested in the following set of log pairs and the corresponding Iitaka fibrations.
	\begin{defn}\label{definition of Dd}
		Fix a DCC set $\Ii\subset [0,1]\cap \mathbb{Q}$, positive integers $n,r,l$, let $\Cd(n,\Ii,l,r)$ be the set of log pairs $(X,\Delta)$ satisfying the following conditions:
		\begin{itemize}
			\item $(X,\Delta)$ is a klt pair of dimension $n$. 
			\item $\mathrm{coeff} \Delta\subset \Ii$.	
			\item $f: X\rightarrow Z$ is the canonical model of $(X,\Delta)$.		
			\item the general fibre $(X_g,\Delta_g)$ of $f$ has a good minimal model.
			\item Let $(Z',B_{Z'}+\M_{Z'})$ be the generalized pair defined in Theorem \ref{canonical bundle formula on fibration}, then $l\M_{Z'}$ is nef and Cartier.
			\item there is a reduced divisor $D$ and a $\mathbb{Q}$-divisor $F\in |K_X+\Delta|_{\mathbb{Q}/Z}$ such that $(X,\Supp(\Delta-F))$ is log smooth over $Z\setminus D$ and $\Ivol(K_X+\Delta+f^*D)\leq r\Ivol(K_X+\Delta)$.	
		\end{itemize}

	\end{defn}
	\begin{thm}
		\label{key theorem}
		Fix a DCC set $\Ii \subset [0,1]\cap \mathbb{Q}$, positive integers $n,r,l$, then the set
		$$\{\Ivol(K_X+\Delta)\ |\ (X,\Delta)\in \Cd(n,\Ii,l,r) \}$$ 
		satisfies the DCC.
	\end{thm}
	
	As an application, we prove the following boundedness result
	\begin{thm}
		\label{key theorem 2}
		Fix a DCC set $\Ii\subset [0,1]\cap \mathbb{Q}$, positive integers $n,r,l$ and a positive rational number $C>0$. Then the set 
		$$\{\mathrm{Proj}R(X,K_X+\Delta)\ |\ (X,\Delta)\in \Cd(n,\Ii,l,r),\ \Ivol(K_X+\Delta)=C \}$$
		is in a bounded family.
	\end{thm}

	The idea is to prove that we can choose an snc model (see definition \ref{snc model}) of $(X,\Delta-F)\rightarrow Z$ to be in a bounded family, this is why we need the last condition in Definition \ref{definition of Dd}. We believe that the existence of $D$ and the integer $r$ naturally comes from the moduli space of the general fiber of $f$. Theorem \ref{fibers in a good moduli space} is an application of Theorems \ref{key theorem} and \ref{key theorem 2} based on this idea. 
	\vspace{0.3cm}

	\section{Prelimiary}
	\textbf{Notation and conventions}.
	Let $\Ii \subset \mathbb{Q}$ be a subset, we say $\Ii$ satisfies the DCC if there is no strictly decreasing subsequence in $\Ii$. For a birational morphism $f: Y\rightarrow X$ and a divisor $B$ on $X$, $f_*^{-1}(B)$ denotes the strict transform of $B$ on $Y$, and $\mathrm{Exc}(f)$ denotes the sum of reduced exceptional divisors of $f$. A fibration means a projective and surjective morphism with connected fibers. For a $\mathbb{Q}$-divisor $D$, a map defined by the linear system $|D|$ means a map defined by $|\lfloor D\rfloor|$. Given two $\mathbb{Q}$-Cartier $\mathbb{Q}$-divisors $A,B$, $A\sim_{\mathbb{Q}} B$ means that there is an integer $m>0$ such that $m(A-B)\sim 0$.
	
	A sub-pair $(X,\Delta)$ consists of a normal variety $X$ and a $\mathbb{Q}$-divisor $\Delta$ on $X$ such that $K_X+\Delta$ is $\mathbb{Q}$-Cartier, we call $(X,\Delta)$ a pair if in addition $\Delta$ if effective. If $g: Y\rightarrow X$ is a birational morphism and $E$ is a divisor on $Y$, the discrepancy $a(E,X,\Delta)$ is $-\mathrm{coeff}_{E}(\Delta_Y)$, where $K_Y+\Delta_Y :=g^*(K_X+\Delta) $. A sub-pair $(X,\Delta)$ is called sub-klt (respectively sub-lc) if for every birational morphism $Y\rightarrow X$ as above, $a(E,X,\Delta)>-1$ (respectively $\geq -1$) for every divisor $E$ on $Y$. A pair $(X,\Delta)$ is called klt (respectively lc) if $(X,\Delta)$ is sub-klt (respectively sub-lc) and $(X,\Delta)$ is a pair. 
	
	Let $(X,\Delta),(Y,\Delta_Y)$ be two sub-pairs and $h:Y\rightarrow X$ a birational morphism, we say $(Y,\Delta_Y)\rightarrow (X,\Delta)$ is a crepant birational morphism if $K_Y+\Delta_Y\sim_{\mathbb{Q}}h^*(K_X+\Delta)$. Two pairs $(X_i,\Delta_i),i=1,2$ are crepant birationally equivalent if there is a sub-pair $(Y,\Delta_Y)$ and two crepant birational morphisms $(Y,\Delta_Y)\rightarrow (X_i,\Delta_i),i=1,2$.
	
	A generalised pair $(X,\Delta+M)$ consists of a normal variety $X$ equipped with a birational morphism $f:X'\rightarrow X$ where $X$ is normal, a $\mathbb{Q}$-boundary $\Delta$, and a $\mathbb{Q}$-Cartier nef divisor $M'$ on $X'$ such that $K_{X}+\Delta+M$ is $\mathbb{Q}$-Cartier and $M=f_*M'$. Let $\Delta'$ be the $\mathbb{Q}$-divisor such that $K_{X'}+\Delta'+M'=f^*(K_X+\Delta+M)$, we call $(X,\Delta+M)$ a generalised klt (respectively lc) pair, if $(X',\Delta')$ is sub-klt (respectively sub-lc).
	
	Let $X\rightarrow Z$ be an algebraic contraction and $R$ a $\mathbb{Q}$-divisor on $X$, we write $R=R_v+R_h$, where $R_v$ is the vertical part and $R_h$ is the horizontal part.
	
	The language of $\mathbf{b}$-divisor was introduced by Shokurov.
	\begin{defn}
		Let $X$ be a projective variety, we say a formal sum $\mathbf{B}=\sum a_\nu \nu$, where the sum ranges over all valuations of $X$, is a $\bm{b}$-divisor, if the set
		$$F_X=\{\nu\ |\ a_\nu \neq 0\text{ and the center $\nu$ on $X$ is a divisor }\},$$
		is finite. The trace $\mathbf{B}_Y$ of $\mathbf{B}$ is the sum $\sum a_\nu B_\nu$, where the sum now ranges over the elements of $F_Y$.
	\end{defn}
	\begin{defn}
		For a klt pair $(X,\Delta)$, $\mu:X\rightarrow U$ a projective morphism such that $K_X+\Delta$ is $\mathbb{Q}$-Cartier, by \cite{BCHM10}, the canonical ring 
		$$R(X/Z,K_X+\Delta):=\oplus_{m\geq 0}\mu_*\mathcal{O}_X(m(K_X+\Delta))$$
		is a finitely generated $\mathcal{O}_U$-algebra. We define the canonical model of $(X,\Delta)$ over $U$ to be $\mathrm{Proj}R(X/U,K_X+\Delta)$. 
	\end{defn}

	Next, we state some results that we will use in what follows.
	
	\begin{thm}[{\cite[Theorem 2.12]{HMX13}}]\label{HX13 Theorem 2.12}
		Let $f:X\rightarrow Z$ be a surjective projective morphism and $(X,\Delta)$ a dlt pair such that
		\begin{itemize}
			\item for a very general point $z\in Z$, the fibre $(X_z,\Delta_z)$ has a good minimal model, and
			\item the ring $R(X/Z,K_X+\Delta)$ is finitely generated.
		\end{itemize}
		Then $(X,\Delta)$ has a good minimal model over $Z$.
	\end{thm}
	\begin{thm} [{\cite[Theorem 1.3]{BZ16}}]\label{thm: gen eff bir}
		Let $d,r$ be two positive integers and $\Ii\subset [0,1]$ a DCC set of real numbers. Then there is a positive number $m_0$ depending only on $d,r$ and $\Ii$ satisfying the following. Assume that
		\begin{itemize}
			\item $(Z,B)$ is a projective lc pair of dimension $d$,
			\item $\mathrm{coeff}(B)\in \Ii$,
			\item $rM$ is a nef Cartier divisor, and
			\item $K_Z+B+M$ is big,
		\end{itemize}
		then the linear system $|m(K_Z+B+M)|$ defines a birational map for every positive integer $m$ such that $m_0\mid m$.
	\end{thm}
	\begin{thm} [{\cite[Theorem 8.1]{BZ16}}]\label{pseudo eff threshold}
		Let $\Ii$ be a DCC set of non-negative real numbers and $d$ a natural number. Then there is a real number $e\in (0,1)$ depending only on $\Ii,d$ such that if
		\begin{itemize}
			\item $(Z,B)$ is projective lc of dimension $d$,
			\item $M=\sum \mu_j M_j$ where $M_j$ are nef Cartier divisors,
			\item the coefficients of $B$ and the $\mu_j$ are in $\Ii$, and
			\item $K_Z+B+M$ is a big divisor,
		\end{itemize}
		then $K_Z+eB+eM$ is a big divisor.
	\end{thm}
	
	\begin{thm} [{\cite[Theorem 1.10]{Fil18}}]\label{DCC for fixed base}
		Let $\Ii \subset [0,1]\cap \mathbb{Q}$ be a DCC set, $(W,D)$ a log smooth pair with $D$ reduced, and $M$ a fixed $\mathbb{Q}$-Cartier $\mathbb{Q}$-divisor on $W$. Suppose $\Cd$ is the set of all projective simple normal crossing pairs $(Z,B)$ such that $\mathrm{coeff}(B)\subset \Ii$, there exists a birational morphism $f: Z\rightarrow W$ and $f_*B\leq D$. Then, the set
		$$\{\mathrm{vol}(K_Z+B+f^*M)\ |\ (Z,B)\in \Cd \}$$
		satisfies the DCC.
	\end{thm}
	\begin{thm}[{\cite[Theorem 1.12]{Fil18}}]\label{Invariance of Plurigenra}
		Let $(\Cz,\Cb)\rightarrow T$ be a log smooth morphism, $\{x_i\}_{i\geq 1}\subset T$ a set of closed points. Denote by $(Z_i,B_i)$ the log pair given by the fiber product $(\Cz,\Cb)\times_T x_i$. Assume that
		\begin{itemize}
			\item $0\leq \Cb\leq \mathrm{red}(\Cb)$, and
			\item there is a $\mathbb{Q}$-divisor $\Cm$ on $\Cz$ such that $M_i=\Cm|_{Z_i}$ is nef for every $i$.
		\end{itemize}
		Then, we have $\mathrm{vol}(Z_i,K_{Z_i}+B_i+M_i)=\mathrm{vol}(Z_j,K_{Z_j}+B_j+M_j)$ for every $i,j\in \mathbb{N}$.
	\end{thm}
	
	The following is a general version of the canonical bundle formula given in \cite{Kol07}.
	\begin{thm}
		\label{canonical bundle formula}
		Let $X,Z$ be normal projective varieties and $f: X\rightarrow Z$ a dominant morphism with generic fiber ${X_\eta}$. Let $R$ be a $\mathbb{Q}$-divisor on $X$ such that $K_X+R$ is $\mathbb{Q}$-Cartier and $B$ a reduced divisor on $Z$. We call $(X,R)\rightarrow Z$ an lc-trivial fibration if
		\begin{itemize}
			\item $K_X+R\sim_{\mathbb{Q},Z}0$,
			\item $h^0({X_\eta},\mathcal{O}_{X_\eta}( \lceil R_{\leq 0}\rceil  ) )=1$, and
			\item $f$ has slc fibers in codimension 1 over $Z\setminus B$, that is, if $D$ is prime divisor not contained in $B$, then
			\begin{itemize}
				\item no component of $R$ dominates $D$, and
				\item $(X,R+f^*D)$ is lc over the generic point of $D$,
			\end{itemize}
		\end{itemize}
		then one can write 
		$$K_X+R\sim_{\mathbb{Q}} f^*(K_Z+B_Z+\M_Z),\ \mathrm{where}$$
		\begin{enumerate}
			\item[(a)] $\M_Z=M(X/Z,R)$ is the moduli part. It is a $\bm{b}$-divisor depending only on  the crepant birationally equivalent class of $({X_\eta},R|_{X_\eta})$ and $Z$ such that:
			\begin{itemize}
				\item there is a birational morphism $Z'\rightarrow Z$ such that $\M_Z$ is the pushforward of $\M_{Z'}:=M(X'/Z',R')$, and $\M_{Z''}=M(X''/Z'',R'')=\pi^*\M_{Z'}$ for any birational morphism $Z''\rightarrow Z'$, where $X'$ is the normalization of the main component of $X\times_Z Z'$ and $(X',R')\rightarrow (X,R)$ is a crepant birational morphism. In this case, we call $\M$ descends on $Z'$.
				\item If $X\rightarrow Z$ and $R,B$ satisfy the standard normal crossing assumption, see Definition \ref{standard normal crossing}, then $\M$ descends on $Z$.
			\end{itemize}
			
			\item[(b)] $B_Z$ is the unique $\mathbb{Q}$-divisor supported on $B$ for which there is a codimension $\geq 2$ closed subset $W\subset Z$ such that:
			\begin{itemize}
				\item $(X\setminus f^{-1}(W),R+f^*(B-B_Z))$ is lc.
				\item every irreducible component of $B$ is dominated by a log canonical centre of $(X,R+f^*(B-B_Z))$.
				
			\end{itemize}
			\item[(c)] If the morphism $X\rightarrow Z$ and divisors $R,B$ satisfy the standard normal crossing assumption, see definition \ref{standard normal crossing}, then $B_Z$ is also the unique $\mathbb{Q}$-divisor such that $ R_v+f^*(B-B_Z)\leq \mathrm{red}(f^*B)$.
		\end{enumerate}
		\begin{proof}
			(a), (b) is {\cite[Theorem 8.5.1]{Kol07}}. 
			
			For (c), if the morphism $X\rightarrow Z$ and $R, B$ satisfy the standard normal crossing assumption, then $(Z, B)$ is log smooth. We replace $R $ by $R+f^*(B-B_Z)$ and $B_Z$ by $B_Z+(B-B_Z)=B$, then, 
			\begin{itemize}
				\item there is a codimension $\geq 2$ closed subset $W\subset Z$ such that $(X\setminus f^{-1}(W),R)$ is sub log canonical, and
				\item every irreducible component of $B$ is dominated by a log canonical centre of $(X,R)$.
			\end{itemize}
			It is easy to see that to prove (c), we only need to prove that $W$ can be chosen to be the empty set, which is equal to saying that $(X, R)$ is lc.
			
			Suppose $(X,R)$ is not log canonical, consider the following diagram
			$$\xymatrix{
				(X',R')  \ar[d] _{f'} \ar[r]^{\pi_X}    &(X,R)\ar[d]^{f}\\
				Z'\ar[r]_{\pi}&    Z.	 
			} $$
			where $\pi$ is birational, $\pi_X$ is crepant birational, $f': X'\rightarrow Z'$ is equidimensional and $\pi_X$ extracts a non-lc place of $(X,R)$, which is denoted by $E$, then $\mathrm{coeff}_E(R')>1$. By the canonical bundle formula,
			$$K_{X'}+R'\sim_{\mathbb{Q}} f'^*(K_{Z'}+B'+\M_{Z'}).$$
			By assumption, $X\rightarrow Z$ and divisors $R,B$ satisfy the standard normal crossing assumption, then $\M$ descends on $Z$ and $\pi^*\M_Z= \M_{Z'}$ and $K_{Z'}+B'\sim_{\mathbb{Q}}\pi^*(K_Z+B)$. Because $(Z,B)$ is lc, $(Z',B')$ is sub-lc.
			
			Let $\tilde{B}$ be a reduced divisor on $Z'$ such that $f'$ has slc fibers in codimension 1 over $Z'\setminus \tilde{B}$. By (b), $B'$ is the unique $\mathbb{Q}$-divisor for which there is a codimension $\geq 2$ closed subset $W'\subset Z'$ such that
			\begin{enumerate}
				\item $(X'\setminus f'^{-1}(W'),R'+f'^*(\tilde{B}-B'))$ is sub-lc, and
				\item every irreducible component of $B'$ is dominated by a log canonical centre of $(X',R'+f'^*(\tilde{B}-B'))$.
				
			\end{enumerate}
			Because $f'$ is equidimensional, $\mathrm{coeff}_E(R'+f'^*(\tilde{B}-B'))\leq 1$ and $\mathrm{coeff}_E(R')>1$, then $\mathrm{coeff}_E(f'^*(\tilde{B}-B'))< 0$. Since $\tilde{B}$ is reduced and $E$ is vertical, then $\mathrm{coeff}_{f'(E)}(B')>1$, which contradicts with that $(Z',B')$ is sub-lc.
		\end{proof}
	\end{thm}
	\begin{defn}[{\cite[Definition 8.3.6]{Kol07}}]\label{standard normal crossing}
		Let $f:X\rightarrow Z$ be a projective morphism between normal projective varieties, $R$ be a $\mathbb{Q}$-divisor on $X$ and $B$ be a divisor on $Z$. We say that $f: X\rightarrow Z$ and $R, B$ satisfy the standard normal crossing assumption if the following hold:
		\begin{enumerate}
			\item $X,Z$ are smooth.
			\item $\Supp (R+f^*B)$ and $B$ are snc divisors.
			\item $f$ is smooth over $Z\setminus B$.
			\item $\Supp R$ is a relative snc divisor over $Z\setminus B$.
		\end{enumerate}
	\end{defn}
	
	In practice, the assumptions on $X$ and the divisors $R, B$ are completely harmless. By contrast, we have to do some work to reduce the problems on $Z$ to the problems on the following "good" birational model of $Z$. 
	\begin{defn}
		\label{snc model}
		An snc model of $f:(X,R)\rightarrow Z$ is a birational model $Z'\rightarrow Z$, such that there is a reduced divisor $D'$ on $Z'$, a $\mathbb{Q}$-divisor $B$ on $Z$, and a crepant birational morphism $\phi: (X', R')\rightarrow (X, R+f^*B)$, such that the morphism $X'\rightarrow Z'$ and $R', D'$ satisfy the standard normal crossing assumption.
		
	\end{defn}
	
	The next theorem says that the canonical bundle formula works on the Iitaka fibration of an lc pair.
	
	\begin{thm}\label{canonical bundle formula on fibration}
		Let $(X,\Delta)$ be a $n$-dimensional klt pair, $f: X\rightarrow Z$ an algebraic contraction such that $\kappa(X_\eta,K_{X_\eta}+\Delta|_{X_\eta})=0$, where $X_\eta$ is the generic fiber of $f$, and a birational morphism $g_W:W\rightarrow Z$. Then there is a commutative diagram
		$$\xymatrix{
			X  \ar[d] _{f}& &   X' \ar[ll]_{{h_X}} \ar[d]^{f'}\\
			Z	&  & 	 {Z'}\ar[ll]^{h}
		} $$
		such that:
		\begin{enumerate}
			
			\item[(1)] $h,{h_X}$ are birational, $h$ factors through $g_W$ and $f'$ is equidimensional.
			\item[(2)] ${Z'}$ is smooth and $X'$ has only quotient singularities.
			\item[(3)] There is a lc pair $K_{X'}+\Delta'$ and a generalized lc pair $({Z'},B_{Z'}+\M_{Z'})$ such that 
			\begin{itemize}
				\item $\M$ descends on $Z'$,
				\item $K_{X'}+\Delta'\sim_{\mathbb{Q}} f'^*(K_{Z'}+B_{Z'}+\M_{Z'})+{F'}$,
				\item $f'_*\mathcal{O}_{X'}(m{F'})\cong \mathcal{O}_{Z'}$, and
				\item $ {h_X}_*\mathcal{O}_{X'}(m(K_{X'}+\Delta'))\cong \mathcal{O}_X(m(K_X+\Delta))$ for all $m\geq 0$.  
			\end{itemize}
			
			\item[(4)] If $\mathrm{coeff}\Delta$ is in a DCC set, and the general fibre $(X_g,\Delta_g)$ of $f$ has a good minimal model, then $ \mathrm{coeff}B_{Z'}$ and $\mathrm{coeff}B_Z$ are in a DCC set, where $B_Z: =g_*B_{Z'}$.
		\end{enumerate}
		\begin{proof}
			By the weak semi-stable reduction Theorem of Abramovich and Karu (cf. \cite{AK00}), we can construct $X',{Z'}$ satisfying (1) and (2), such that $\M$ descends on $Z'$.
			
			For (3), let $\Delta': ={h_X} _* ^{-1}\Delta+E$, where $E$ is the exceptional divisor. Because $(X,\Delta)$ is klt, it is easy to see that ${h_X}_*\mathcal{O}_{X'}(m(K_{X'}+\Delta'))\cong \mathcal{O}_X(m(K_X+\Delta))$ for all $m\geq 0$, and $\kappa(X,K_X+\Delta)=\kappa(X',K_{X'}+\Delta')$. 
			
			If $\kappa(X,K_X+\Delta)<0$, choose $a\gg 0$ such that $a(K_{X'}+\Delta')$ is Cartier. Since ${Z'}$ is smooth and $f'$ is equidimensional, $f'_*\mathcal{O}_{X'}(a(K_{X'}+\Delta'))$ is a line bundle on $X'$, denote it by $\mathcal{O}_{Z'}(D)$. Choose a general sufficiently ample divisor $A'$ on ${Z'}$ such that $\mathcal{O}_{Z'}(A'+D)$ is big. Let $A: =\frac{1}{a}A'$, then $f'_*\mathcal{O}_{X'}(a(K_{X'}+\Delta'+f'^*A))$ is a big line bundle, and $\kappa(X',K_{X'}+\Delta'+f'^*A)=\mathrm{dim}({Z'})\geq 0$. Because $A'$ is general, $(X',\Delta'+f'^*A)$ is lc. It is easy to see that to prove (3) we may replace $(X',\Delta')$ by $(X',\Delta'+f'^*A)$ and assume $\kappa(X, K_X+\Delta)\geq 0$.
			
			Suppose $\kappa(X,K_X+\Delta)\geq 0$, choose $a\gg 0$ such that $H^0(X',\mathcal{O}_{X'}(a(K_{X'}+\Delta')))>0$, then we can choose $L\in |a(K_{X'}+\Delta')|$. Define
			\begin{equation*}
				G: =\mathrm{max}\{ N\ |\ N\mathrm{\ is\ an\ effective\ }\mathbb{Q}\text{-divisor such that }L\geq f'^* N \},
			\end{equation*}
			and
			\begin{equation*}
				D:=\frac{1}{a}G,\text{ and } {F'}:=\frac{1}{a}(L-f'^*G),
			\end{equation*}
			then we have $K_{X'}+\Delta'\sim_{\mathbb{Q}} f'^*D+{F'}$. 
			Because ${Z'}$ is smooth and $f'$ is equidimensional, $f'_*\mathcal{O}_{X'}(m{F'})$ is an invertible sheaf. Since $\Supp({F'})$ does not contain the whole fiber over any codimension 1 point on ${Z'}$, it is easy to see that $f'_*\mathcal{O}_{X'}(m{F'})\cong \mathcal{O}_{Z'}$ for all $m\geq 0$. 
			
			Let $X_\eta'$ be the generic fiber of $f'$, then $(K_{X'}+\Delta')|_{X_\eta'}={F'}|_{X_\eta'}$. Because $f'_*\mathcal{O}_{X'}(m{F'})\cong \mathcal{O}_{Z'}$, we have $H^0(X_\eta',\mathcal{O}_{X_\eta'}( \lceil (\Delta'-{F'})_{\leq 0}\rceil  ) )=1$, and $f':(X',\Delta'-{F'})\rightarrow Z'$ is an lc-trivial fibration.
			By the canonical bundle formula, there is a generalized pair $({Z'}, B_{Z'}+\M_{Z'})$ such that 
			$$K_{X'}+\Delta'-{F'}\sim_{\mathbb{Q}}f'^*(K_{Z'}+B_{Z'}+\M_{Z'}).$$ 
			Also because $f'_*\mathcal{O}_{X'}(m{F'})\cong \mathcal{O}_{Z'}$, there is an integer $l>0$ such that 
			$$H^0(X',\mathcal{O}_{X'}(ml(K_{X'}+\Delta')))\cong H^0({Z'},\mathcal{O}_{Z'}(ml(K_{Z'}+B_{Z'}+\M_{Z'})))$$
			for all $m\geq 0$.

			For (4), because $(X,\Delta)$ is a klt pair, by the main theorem of \cite{BCHM10}, $R(X/Z,K_X+\Delta)$ is finitely generated. Also because the general fiber $(X_g,\Delta_g)$ has a good minimal model, by Theorem \ref{HX13 Theorem 2.12}, $(X,\Delta)$ has a good minimal model over $Z$, we denote it by $(Y,\Delta_Y)$, and the morphism $f_Y:Y\rightarrow Z$. Because $\kappa(X_\eta,K_{X_\eta}+\Delta_{X_\eta})=0$, then $K_Y+\Delta_Y\sim_{\mathbb{Q},Z}0$.
			
			If $\mathrm{coeff}\Delta$ is in a DCC set, by the construction of $\Delta_Y$, $\mathrm{coeff}\Delta_Y$ is also in a DCC set. Let $B'$ be the divisor on $Z'$ such that $f'$ has slc fibers in codimension 1 over $Z'\setminus B'$. By construction, there is a codimension $\geq 2$ closed subset $W'\subset Z'$ such that $B_{Z'}$ is the smallest $\mathbb{Q}$-divisor supported on $B'$ such that $(X'\setminus f'^{-1}(W'),\Delta'-{F'}+f'^*(B'-B_{Z'}))$ is lc. 
			
			Denote the rational contraction $X'\dashrightarrow Y$ by $h_Y$, let $W:=h(W')$, because $K_{X'}+\Delta'-{F'}\sim_{\mathbb{Q},Z}0$, $\Delta_Y$ is the pushforward of $\Delta'$ on $Y$ and $K_Y+\Delta_Y\sim_{\mathbb{Q},Z}0$, then $B_{Z'}$ is also the smallest $\mathbb{Q}$-divisor supported on $B'$ such that $(Y\setminus f_Y^{-1}(W),\Delta_Y+{h_Y}_*f'^*(B'-B_{Z'}))$ is lc. Because $\mathrm{coeff}\Delta_Y$ is in a DCC set, then by the main result of \cite{HMX14}, $\mathrm{coeff}(B_{Z'})$ is in a DCC set.
		\end{proof}
	\end{thm}
	
	\section{DCC of Iitaka volumes}

	Suppose $X\dashrightarrow Z$ is the canonical model of $(X,\Delta)$. By Theorem \ref{canonical bundle formula on fibration}, there is a positive integer $l$, a nef $\mathbb{Q}$-Cartier $\mathbb{Q}$-divisor $\M_{Z'}$ and a $\mathbb{Q}$-divisor $B_{Z'}$ on a birational model ${Z'}$ of $Z$ such that 
	$$H^0(X,\mathcal{O}_X(ml(K_X+\Delta)))\cong H^0({Z'},\mathcal{O}_{Z'}(ml(K_{Z'}+B_{Z'}+\M_{Z'})))$$
	for all $m\geq 0$. Then the canonical model of $(X,\Delta)$ is isomorphic to the canonical model of $({Z'},B_{Z'}+\M_{Z'})$, and we have 
	$$\Ivol(K_X+\Delta)=\mathrm{vol}(Z,K_Z+B_Z+\M_Z)=\mathrm{vol}({Z'},K_{Z'}+B_{Z'}+\M_{Z'}),$$
	where $B_Z$ and $\M_Z$ are the pushforward of $B_{Z'}$ and $\M_{Z'}$ on $Z$. To study the Iitaka volume of $K_X+\Delta$, we only need to work on the birational model ${Z'}$ of $Z$.

	\begin{lemma}
		\label{boundedness of moduli part}
		Fix a positive integer $C$ and a finite set $\Ii\subset [0,1]\cap \mathbb{Q}$. Suppose there is a smooth projective morphism $\Cz\rightarrow T$ between schemes of finite type and a relative very ample divisor $\mathcal{A}$ on $\Cz$ over $T$. Let $\Cs$ be a set of generalized pairs, such that for every $(Z,B_Z+\M_Z)\in \Cs$, there is a closed point $t\in T$ such that
		\begin{itemize}
			\item there is a birational morphism $\phi:Z\rightarrow \Cz_t$, and
			\item $\phi_*\M_Z\sim_{\mathbb{Q}}D_1-D_2$ for two effective $\mathbb{Q}$-divisors $D_k$ with $\mathrm{coeff}(D_k)\subset \Ii$ and $\mathrm{deg}_{\mathcal{A}_t}(D_k)\leq C$, $k=1,2$.
		\end{itemize}
		Then there is a smooth projective morphism $\Cz'\rightarrow T'$ and finitely many $\mathbb{Q}$-divisors $\Cm_k$ on $\Cz'$ over $T'$ such that for any $(Z,B_Z+\M_Z)\in \Cs$, there is a closed point $t'\in T'$ and an isomorphism $\psi:\Cz_t\rightarrow \Cz'_{t'} $, such that $ \psi_*\phi_*\M_Z\sim_{\mathbb{Q}}\Cm_k|_{\Cz'_{t'}}$.
		\begin{proof}
			Since the coefficients of $D_k,k=1,2$ are in a finite set $\Ii$, there is a positive number $\delta$ such that $\mathrm{coeff}(D_k)>\delta$, which implies $\frac{1}{\delta}D_k\geq \lfloor D_k\rfloor,k=1,2$. Because $\mathrm{deg}_{\mathcal{A}_t}(D_k)\leq C$, then $\mathrm{deg}_{\mathcal{A}_t}(\Supp(D_k))\leq \frac{C}{\delta}$. By the boundedness of the Chow variety, there is a morphism $\Cz'\rightarrow T'$ and a divisor $\mathcal{D}$ on $\Cz'$, such that for every closed point $t\in T$, there is a closed point $t'\in T'$ and an isomorphism $\psi:\Cz_{t'}\rightarrow \Cz'_t$, such that $\Supp(\psi_*D_k)\subset \mathcal{D}|_{\Cz'_{t'}},k=1,2$. 
			
			Let $R$ be a component of $\mathcal{D}$, $S\rightarrow T'$ the normalisation of the Stein factorisation of $R\rightarrow T'$, such that $S\rightarrow T' $ is finite and $S$ is normal, and  
			$$\xymatrix{
				\Cz''  \ar[d] \ar[r]&    \Cz' \ar[d]\\
				S  \ar[r]	&   	 T',
			} $$
			the diagram. Because $S\rightarrow T'$ is finite, $S$ is irreducible and $\Cz''\rightarrow S$ is flat, $\Cz''$ is a quasi-projective variety and $\Cz''$ is normal by \cite[5.12.7]{Gro66}. Replacing $\Cz'\rightarrow T'$ by $\Cz''\rightarrow S$ finitely many times, we may assume that the fibers of $R\rightarrow T'$ are irreducible, for every component $R$ of $\mathcal{D}$. 
			
			Since, for every component $R$ of $\mathcal{D}$, the coefficients of $R$ in $D_1, D_2$ are in a finite set, then there are only finitely many possibilities for $D_1, D_2$ and $D_1-D_2$. Then there are only finitely many $\mathbb{Q}$-divisors $\Cm_k$ on $\Cz'$ over $T'$ such that $\psi_*\phi_*\M_Z\sim_{\mathbb{Q}}\Cm_k|_{\Cz_t}$ for some $k$. 
		\end{proof}
	\end{lemma}

	The next theorem says that if we bound the Iitaka volume of $(X,\Delta)\in \Cd(n,\Ii,l,r)$, then we can choose the snc model of $X\rightarrow \mathrm{Proj}R(X,K_X+\Delta)$ to be in a bounded family depending only on $n,\Ii,l$ and $r$. 
	
	\begin{thm}\label{boundedness of Ambro model}
		Fix a DCC set $\Ii\subset [0,1]\cap \mathbb{Q}$, positive integers $n,l,r,C>0$, and a positive number $\delta>0$. Define $\Cd'(n,\Ii,l,r,C,\delta)$ to be the set of $n$-dimensional log pairs $(X,\Delta)$ with the following properties:
		\begin{itemize}
			\item $(X,\Delta)\subset \Cd(n,\Ii,l,r)$. 
			\item $\Ivol(K_X+\Delta)\leq C$.	
			\item there is an effective ample $\mathbb{Q}$-divisor $H$ on $Z$ with $\mathrm{coeff}(H)>\delta$ such that $\Ivol(K_X+\Delta+f^*H)\leq r\Ivol(K_X+\Delta)$.
		\end{itemize}
		Then there is a log pair $(\Cz,\pp)$ with a projective morphism $\Cz\rightarrow T$, where $T$ is of finite type, and finitely many $\mathbb{Q}$-divisors $\Cm_k,k\in \Lambda$ on $\Cz$, such that for every $(X,\Delta) \in \Cd'(n,\Ii,l,r,C,\delta)$, there is a closed point $t\in T$, such that 
		\begin{itemize}
			\item $\Cz_t$ is birationally equivalent to the canonical model of $(X,\Delta)$,
			\item if $\M_{\Cz_t}:=M(X/\Cz_t,\Delta-F)$ is the moduli part, then $\M_{\Cz_t}\sim_{\mathbb{Q}}\Cm_{k}|_{\Cz_t}$ for some $k\in \Lambda$, 
			\item there is birational morphism $X'\rightarrow X$ and a $\mathbb{Q}$-divisor $F'$ on $X'$, such that the morphism $X'\rightarrow \Cz_t$ and $\Delta'-F',\pp_t$ satisfy the standard normal crossing assumption, where $\Delta'$ is the strict transform of $\Delta$ plus the exceptional divisor and $F'\in |K_{X'}+\Delta'|_{\mathbb{Q}/\Cz_t}$ is a section, in particular, $\M$ descends on $\Cz_t$, and
			\item if $(Z',B_{Z'}+\M_{Z'})$ is the generalised pair defined in Theorem \ref{canonical bundle formula on fibration} such that there is a birational morphism $\phi_t:Z'\rightarrow \Cz_t$, then $B:=\phi_{t*}B_{Z'}\leq \pp_t$.
		\end{itemize}  
		
		\begin{proof}
			Suppose $\kappa(X,K_X+\Delta)=d\leq n$, then $\mathrm{dim}(Z)=d$.
			Let $(Z',B_{Z'}+\M_{Z'})$ be a generalized pair given in Theorem \ref{canonical bundle formula on fibration} and denote the morphism $Z'\rightarrow Z$ by $h$. By (4) of Theorem \ref{canonical bundle formula on fibration}, $\mathrm{coeff}(B_{Z'})$ is in a DCC set $\Ii'$ depending only on $\Ii,d,n$. By assumption, $l\M_{Z'}$ is nef and Cartier. 
			
			By Theorem \ref{thm: gen eff bir}, there is an integer $r'$, such that $|r'(K_{Z'}+B_{Z'}+\M_{Z'})|$ defines a birational map. Perhaps replacing $Z'$ by a birational model, we may assume that it defines a birational morphism $\phi: Z'\rightarrow W$. Then there is a very ample divisor $A$ on $W$, such that $ r'(K_{Z'}+B_{Z'}+\M_{Z'})=\phi^*A+F_1$, where $F_1$ is an $\phi$-effective $\mathbb{Q}$-divisor. Because
			$$A^d \leq \mathrm{vol}(r'(K_{Z'}+B_{Z'}+\M_{Z'}))
			=r'^d\Ivol(K_X+\Delta)
			<r'^dC,$$
			then $W$ is in a bounded family. Then there exists a projective morphism $\Cw'\rightarrow T$ and a relative very ample divisor $\mathcal{A}'$ depending only on $n,\Ii,l,C$, such that there is a closed point $t\in T$ and an isomorphism $\chi:W\rightarrow  \Cw'_t$ such that $\chi^*\mathcal{A}'_t= A$. 
			Because $r'$ is fixed and the coefficients of $B_{Z'}$ are in a DCC set $\Ii'$, it is easy to see that the coefficient of $F_1$ is also in a DCC set $\tilde{\Ii}$. 
			
			Passing to a stratification of $T$ and a log resolution of the generic fiber of $\Cw'\rightarrow T$, we may assume that there is a birational morphism $\xi:\Cw\rightarrow \Cw'$, and $\Cw\rightarrow T$ is a smooth morphism. Let $\mathcal{A}$ be a very ample divisor on $\Cw$ over $T$, then there is an integer $r''$ such that $r''\xi^*\mathcal{A}'-\mathcal{A}$ is big over $T$. After increasing $r'$, replacing $Z'$ by a birational model and $(W,A)$ by $(\Cw_t,\mathcal{A}_t)$, we may assume $W$ is smooth and there is a very ample divisor $A$ on $W$ such that $A^d\leq \mathrm{vol}(r'(K_{Z'}+B_{Z'}+\M_{Z'}))$.

			Let $m$ be the Cartier index of $H$, define 
			\begin{equation}
				\begin{aligned}
					L':&=\frac{1}{r'}(\phi^*A+F_1)+(2d+1)\phi^* A+(2d+1)mh^*H\\
					&\sim_{\mathbb{Q}} K_{Z'}+B_{Z'}+\M_{Z'}+(2d+1)\phi^*A+(2d+1)mh^*H 
				\end{aligned}
			\end{equation}
			Because $H$ is an effective ample $\mathbb{Q}$-divisor on $Z$, by \cite{BCHM10}, the canonical model of $K_{Z'}+B_{Z'}+\M_{Z'}+(2d+1)\phi^*A+(2d+1)mh^*H$ exists, denote it by $h':Z'\dashrightarrow Z^!$, then $h'_*(K_{Z'}+B_{Z'}+\M_{Z'}+(2d+1)\phi^*A+(2d+1)mh^*H)\sim_{\mathbb{Q}}h'_*L'$ is ample, write $L^!:=h'_*L'$. Because $\phi^*A$ and $mh^*H$ are nef Cartier divisors, by \cite[Lemma 4.4]{BZ16}, both $\phi^*A$ and $h^*H$ are $h'$-trivial, so there are two birational morphisms ${\phi^!}:Z^!\rightarrow W$ and $h^!:Z^!\rightarrow Z$. 
			
			$$\xymatrix{
				& &Z' \ar[dll]_{h} \ar[drr]^{\phi}	\ar@{-->}[d]^{h'} & &  \\
				Z	&&  Z^! \ar[ll]_{h^!} \ar[rr]^{{\phi^!}}&& 	 W
			} $$
			Because $L^!$ is ample and effective, $W$ is smooth, by the negativity lemma, $L^!={\phi^!}^*{\phi^!}_*L^!-F_W$, where $F_W$ is effective and has the same support as $\mathrm{Exc}({\phi^!})$. Then we have $\Supp({\phi^!}^*{\phi^!}_*L^!)\supset \mathrm{Exc}({\phi^!})$, and $Z^!\setminus \Supp(L^!)\supseteq W\setminus \Supp({\phi^!}_* L^!)$. 
			
			Consider the following diagram\\
			$$\xymatrix{
				X  \ar[d] _{f}& &   X^! \ar[ll]_{g} \ar[d]^{f^!}\\
				Z	&  & 	 Z^!\ar[ll]^{h^!}
			} $$
			where $X^!$ is the normalization of the main component of $X\times_Z Z^!$. Because $\Supp(h^{!*}D)=h^{!-1}(\Supp(D))$
			and $X\rightarrow Z$ is smooth over $Z\setminus{\Supp(D)}$, $X^!\rightarrow Z^!$ is smooth over $Z^!\setminus \Supp(h^{!*}D)$. Because $Z^!$ is normal, then $f^{!-1}(Z^!\setminus \Supp(h^{!*}D))$ is normal and
			$$f^{!-1}(Z^!\setminus \Supp(h^{!*}D))\cong f^{-1}(Z\setminus \Supp(D)) \times_{Z\setminus \Supp(D)} Z^!\setminus \Supp(h^{!*}D).$$
			
			Let $K_{X^!}+\Delta^{!'}-F^{!'}:=g^*(K_X+\Delta-F)$, where $\Delta^{!'},F^{!'}$ are two effective $\mathbb{Q}$-divisors with no common component. Suppose $\Delta^{!'}=\Delta^{!''}+\Delta_1$, $F^{!'}=F^{!''}+F_1$, where $\Delta_1,F_1$ are $f^!$-vertical and not supported on $ f^{!-1}(\Supp(h^{!*}D))$, $\Delta^{!''}, F^{!''}$ are either $f$-horizontal or supported on $f^{!-1}(\Supp(h^{!*}D))$. Because $f^!$ is smooth over $ Z^!\setminus \Supp(h^{!*}D)$, then $\Delta_1|_{f^{!-1}(Z^!\setminus \Supp(h^{!*}D))}$ and $F_1|_{f^{!-1}(Z^!\setminus \Supp(h^{!*}D))}$ are the pullback of two divisors on $Z^!\setminus \Supp(h^{!*}D)$. It is easy to see that there is a $\mathbb{Q}$-divisor $R^!$ on $X^!$ such that $ \Supp(R^!)\subset f^{!-1}(\Supp(h^{!*}D))$ and $\Delta_1-F_1-R^!\sim_{\mathbb{Q},f^!}0$, let $\Delta^!-F^!= \Delta^{!''}-F^{!''}+R^!$, then $K_{X^!}+\Delta^!-F^!\sim_{\mathbb{Q},f^!} 0 $. 
			
			If $P$ is a component of $\Supp(\Delta^!-F^!)$, then it is either supported on $f^{!-1}(h^{!*}D)$ or $f^!$-horizontal, and $\Supp(\Delta^!-F^!)$ does not contain the whole fiber of any prime $g$-exceptional divisor that is not contained in $f^{!-1}(h^{!*}D) $. We have
			$$ \Supp(\Delta^!-F^!)|_{ Z^!\setminus \Supp(h^{!*}D)} = \Supp(\Delta-F) \times_{Z}  Z^!\setminus \Supp(h^{!*}D),$$ 
			and it is easy to see that $f^!:(X^!,\Supp(\Delta^!-F^!))\rightarrow       Z^!$ is log smooth over $Z^!\setminus \Supp(h^{!*}D)$.

			Because
			$$Z^! \setminus \Supp(h^{!*}D) 
			\supseteq Z^!\setminus\{\Supp(L^!+ h^{!*}D)\}
			\supseteq W\setminus\{\Supp({\phi^!}_*L^!+ {\phi^!}_*h^{!*}D)\}, $$
			$X^!\rightarrow W$ is isomorphic to $X^!\rightarrow Z^!$ over $W\setminus\{\Supp({\phi^!}_*L^!+ {\phi^!}_*h^{!*}D)\}$. Then $(X^!,\Supp(\Delta^!-F^!))\rightarrow W$ is log smooth over $W\setminus\{\Supp({\phi^!}_*L^!+ {\phi^!}_*h^{!*}D)\}$. 
			
			Next we prove that $(W,\Supp({\phi^!}_*L^!+B_W+ {\phi^!}_*h^{!*}D))$ is log bounded, where $B_W:=\phi_*B_{Z'}$. Because   
			$\Supp({\phi^!}_*L^!+B_W+{\phi^!}_*h^{!*}D)=\Supp(\phi_*(\phi^*A+F_1+B_{Z'}+h^*(D+H))),$
			we only need to prove that $(W,\Supp(\phi_*(\phi^*A+F_1+B_{Z'}+h^*(D+H))))$ is log bounded. Recall that the coefficients of $F_1$ and $B_{Z'}$ are in a DCC set and $\mathrm{coeff}(H)\geq \delta$ by assumption, then there is a positive number $\delta' < 1$, such that $\frac{1}{\delta'}(F_1+B_{Z'}) \geq \mathrm{red}(F_1+B_{Z'})$. By assumption, $A$ and $D$ are two effective integral divisors, and so we only need to prove that there exists a constant $C'>0$ such that 
			$$A^{d-1}.\phi_*(\mathrm{red}(\phi^*A+F_1+B_{Z'}+h^*D)+h^*H)<C'.$$	
			By the projection formula, this is equal to
			$$(\phi^*A)^{d-1}.(\mathrm{red}(\phi^*A+F_1+B_{Z'}+h^*D)+h^*H)<C'.$$

			Let $G=2((2d+1)+1)\phi^*A$, by \cite[Lemma 3.2]{HMX13}, we have
			\begin{equation}
				\begin{aligned}
					&G^{d-1}.(\mathrm{red}(\phi^*A+F_1+B_{Z'}+h^*D)) \\
					\leq &2^d\mathrm{vol}(K_{Z'}+\frac{1}{\delta'}B_{Z'}+\phi^*A+\frac{1}{\delta'}F_1+h^*D+G).
				\end{aligned}
			\end{equation}
			By Theorem \ref{pseudo eff threshold}, there is a positive number $e<1$ such that $K_{Z'}+eB_{Z'}+\M_{Z'}$ is big. Because $\M_{Z'}$ is pseudo-effective and 
			$$K_{Z'}+\frac{1}{\delta'}B_{Z'} +\frac{\frac{1}{\delta'}-1}{1-e}(K_{Z'}+eB_{Z'}+\M_{Z'})+\M_{Z'} \sim_{\mathbb{Q}} \frac{\frac{1}{\delta'}-e}{1-e}(K_{Z'}+B_{Z'}+\M_{Z'}),$$		
			then for any divisor $E$, we have that $\mathrm{vol}(Z',E+K_{Z'}+\frac{1}{\delta'}B_{Z'})\leq \mathrm{vol}(Z',E+ \frac{\frac{1}{\delta'}-e}{1-e}(K_{Z'}+B_{Z'}+\M_{Z'}))$. Because $\phi^*A$ and $F$ are effective, it is easy to see that $\phi^*A+\frac{1}{\delta'}F+G\leq (1+2((2d+1)+1)+\frac{1}{\delta'})(\phi^*A+F)\sim_{\mathbb{Q}} (1+2((2d+1)+1)+\frac{1}{\delta'})r'(K_{Z'}+B_{Z'}+\M_{Z'})$, then
			\begin{equation}
				\begin{aligned}
					&\mathrm{vol}(K_{Z'}+\frac{1}{\delta'}B_{Z'}+\phi^*A+\frac{1}{\delta'}F_1+h^*D+G)\\
					\leq &\mathrm{vol}(\frac{\frac{1}{\delta'}-e}{1-e}(K_{Z'}+B_{Z'}+\M_{Z'})+\phi^*A+\frac{1}{\delta'}F_1+h^*D+G)\\
					\leq &\mathrm{vol}((\frac{\frac{1}{\delta'}-e}{1-e}+r'+2r'((2d+1)+1)+\frac{r'}{\delta'})(K_{Z'}+B_{Z'}+\M_{Z'}) +h^*D)\\
					\leq & C''^d\mathrm{vol}(K_{Z'}+B_{Z'}+\M_{Z'}+h^*D)
				\end{aligned}
			\end{equation}
			where $C'':=(\frac{\frac{1}{\delta'}-e}{1-e}+r'+2r'((2d+1)+1)+\frac{r'}{\delta'})^d$.
			
			Recall that by construction, $H^0(X,m(K_X+\Delta +f^*D))\cong H^0(Z',m(K_{Z'}+B_{Z'}+\M_{Z'}+h^*D))$ 
			for all $m\gg 0$ sufficiently divisible, then we have 
			$$\mathrm{vol}(K_{Z'}+B_{Z'}+\M_{Z'}+h^*D)= \Ivol(K_X+\Delta+f^*D)\leq r\Ivol(K_X+\Delta)\leq rC,$$
			where the second inequality is from the definition of $\Cd(n,\Ii,l,r)$. Then we have
			$$G^{d-1}.(\mathrm{red}(\phi^*A+F_1+B_{Z'}+h^*D))\leq 2^drC''C.$$	
			Because $\phi^*A$ and $h^*H$ are nef, we have that
			\begin{equation}
				\begin{aligned}
					(\phi^*A)^{d-1}.h^*H&\leq (\phi^*A+h^*H)^d\leq r'^d \mathrm{vol}(K_{Z'}+B_{Z'}+\M_{Z'}+H)\\
					&\leq r'^d \Ivol(K_X+\Delta+f^*H)\leq r'^drC.
				\end{aligned}
			\end{equation}
			Let $C':=2^dC''C/(2(2d+1)+1)^d+r'^drC$, then $(\phi^*A)^{d-1}.(\mathrm{red}(\phi^*A+F_1+B_{Z'}+h^*D)+h^*H)<C'$. By the boundedness of the Chow varieties, $(W,\phi_*(\phi^*A+F_1+h^*D))$ is log bounded.
			
			By the definition of log boundedness, there is a flat morphism $(\Cz,\pp)\rightarrow T$ such that $(W,\phi_*(\phi^*A+F_1+h^*D))\cong (\Cz_t,\pp_t)$ for a closed point $t\in T$. Because $f^!:(X^!,\Supp(\Delta^!-F^!))\rightarrow W $ is log smooth over $W\setminus\Supp({\phi^!}_*L^!+ {\phi^!}_*h^{!*}D)$ and $ \Supp({\phi^!}_*L^!+ {\phi^!}_*h^{!*}D)=\Supp(\phi_*(\phi^*A+F_1+h^*D))$, then there is a rational contraction $f_t:X^!\dashrightarrow \Cz_t$, such that $(X^!,\Delta^!-F^!)$ is log smooth over $\Cz_t\setminus \pp_t$. After passing to a stratification of $T$ and log resolution of the generic fiber of $\Cz\rightarrow T$, we can assume $(\Cz_t,\pp_t)$ is log smooth for every closed point $t\in T$. Therefore, for any $(X,\Delta)\in \Cd'(n,\Ii,l,r,C,\delta)$, after replacing $X$ by a birational model, we may assume there is a morphism $X\rightarrow \Cz_t$ for a closed point $t\in T$, such that the morphism $X\rightarrow \Cz_t$ and divisors $\Delta-F,\pp_t$ satisfy the standard normal crossing assumptions. Hence the corresponding moduli part $\M$ descends on $\Cz_t$. Because $X\rightarrow Z$ has the same generic fiber as $f:X\rightarrow Z'$, $\M_{\Cz_t}$ is equal to $\phi_*\M_{Z'}$.  
			
			By Theorem \ref{pseudo eff threshold}, there is a rational number $e<1$ depending only on $\Ii,d,l$ such that both $K_{Z'}+B_{Z'}+\M_{Z'}$ and $K_{Z'}+B_{Z'}+e\M_{Z'}$ are big divisors. By Theorem \ref{thm: gen eff bir}, there is an integer $\tilde{r}$ depending only on $\Ii,d,l,e$, such that both $|m\tilde{r}(K_{Z'}+B_{Z'}+\M_{Z'})|$ and $|m\tilde{r}(K_{Z'}+B_{Z'}+e\M_{Z'})|$ define birational maps for all integers $m\geq 1$. By assumption, $l\M_{\Cz_t}$ is Cartier, we may choose $\tilde{r}=r'l$ for some integer $r'\gg 0$ such that both $\tilde{r}\M_{Z'}$ and $\tilde{r}e\M_{Z'}$ are Cartier divisors and both  
			$| \tilde{r}(K_{Z'}+\lfloor \tilde{r}B_{Z'}\rfloor / \tilde{r}+\M_{Z'})|$ and $| \tilde{r}(K_{Z'}+\lfloor \tilde{r}B_{Z'}\rfloor / \tilde{r}+e\M_{Z'})|$ define birational maps. Let $D_1''\in|\tilde{r}(K_{Z'}+\lfloor \tilde{r}B_{Z'}\rfloor /\tilde{r}+\M_{Z'})|,D_2''\in |\tilde{r}(K_{Z'}+\lfloor \tilde{r}B_{Z'}\rfloor /\tilde{r}+e\M_{Z'})|$ be general members. Define two effective $\mathbb{Q}$-divisors, 
			$$D_1'\sim_{\mathbb{Q}}\frac{K_{Z'}+\lfloor \tilde{r}B_{Z'}\rfloor /\tilde{r}+\M_{Z'}}{1-e},\ D_2'\sim_{\mathbb{Q}}\frac{K_{Z'}+\lfloor \tilde{r}B_{Z'}\rfloor /\tilde{r}+e\M_{Z'}}{1-e}.$$ 
			It is easy to see that the coefficients of $D'_1$ and $D'_2$ are in a discrete set that depends only on $r, \tilde{r},e, \Ii$. Let $ D_1=\phi_* D_1',D_2=\phi_*D_2'$, it is easy to see the degree of $D_1, D_2$ with respect to $A$ in $W$ is bounded. Because the coefficients of $D_1, D_2$ are in a finite set and $\M_{Z}=D_1-D_2$, then by Lemma \ref{boundedness of moduli part}, there are finitely many divisors $\Cm_k,k\in \Lambda$ on $\Cz$ such that $\M_{\Cz_t}=\Cm_k|_{\Cz_t}$ for some $k\in \Lambda$.
			
		\end{proof}
		
	\end{thm}

	\begin{thm}
		\label{boundedness of Ambro model means DCC}
		Suppose $( \Cz,\pp)$ is a log pair, $\Cz\rightarrow T$ is a smooth projective morphism and $\pp$ is log smooth over $T$, where $T$ is of finite type. Let $\Cm_k,k\in \Lambda$ be finitely many $\mathbb{Q}$-divisors. Fix an integer $l>0$ and a DCC set $\Ii\subset [0,1]\cap \mathbb{Q}$. For a closed point $t\in T$, let $\Cd'_t$ denotes the set of generalized log pairs $(W,B_W+\M_W)$ such that
		\begin{itemize}
			\item $(W,B_W)$ is log smooth, 
			\item there is a birational morphism $\phi:W\rightarrow \Cz_t$,
			\item $\M_W=\phi^*(\Cm_k|_{\Cz_t})$, for some $k\in \Lambda$, and
			\item $\mathrm{coeff}(B_W)\in \Ii$ and $\phi_* B_W \leq \pp|_{\Cz_t}$.
		\end{itemize}
		Then the set $\{\mathrm{vol}(K_W+B_W+\M_W)\ |\ (W,B_W+\M_W)\in \cup_{t\in T}\Cd'_t \} $ satisfies the DCC.
		\begin{proof}
			Because $\Lambda$ is a finite set, to prove the DCC, we may assume that $ \Lambda=\{1 \}$. 
			
			Fix a closed point $t\in T$, for any $(W,B_W+\M_W)\in\Cd_t$ such that $\M_W=\phi^*\Cm_1|_{\Cz_t}$, because $(\Cz_t,\pp_t)$ is log smooth, by the proof of \cite[Theorem 1.9]{HMX13}, we may assume that $\phi:W\rightarrow \Cz_t$ only blows up strata of $\pp_t$. On the other hand, because $(\Cz,\pp)$ is log smooth over $T$, and by the proof of Lemma \ref{boundedness of moduli part}, perhaps replacing $T$ by an \'etale cover, every stratum of $(\Cz,\pp)$ have irreducible fibers over $T$, we may find a sequence of blows up $g:\Cz'\rightarrow \Cz$ such that $W=\Cz'_t$. It is easy to see that there is a unique divisor $\mathcal{B}_W$ supported on the strict transform of $\pp$ and the exceptional locus of $g$, such that $B_W=\mathcal{B}_W|_{\Cz'_t}$. Let $Y=\Cz'_0$ be the fiber over $0$ of $\Cz'\rightarrow T$ and $B_Y:=\mathcal{B}_W|_{\Cz'_0}$. By Theorem \ref{Invariance of Plurigenra}, we have that $$\mathrm{vol}(K_W+B_W+\phi^*\Cm_1|_{\Cz_t})=\mathrm{vol}(K_Y+B_Y+(g^*\Cm_1)|_{\Cz'_0}).$$ 
			Then the set
			$$\{\mathrm{vol}(K_W+B_W+\M_W)\ |\ (W,B_W+\M_W)\in \Cd'_t\}$$
			is independent of $t\in T$. Now apply Theorem \ref{DCC for fixed base}.
		\end{proof}
	\end{thm}

	\begin{proof}[Proof of Theorem \ref{key theorem}]
		Fix an arbitrary constant $C>0$, let 
		$$\Cd(n,\Ii,l,r,C^-):=\{(X,\Delta)\in \Cd(n,\Ii,l,r),\Ivol(K_X+\Delta)\leq C \},$$
		we only need to prove $\{\Ivol (K_X+\Delta)\ |\ (X,\Delta)\in \Cd(n,\Ii,l,r,C^-) \}$ is a DCC set. 
		
		Fix $(X,\Delta)\in \Cd(n,\Ii,l,r,C^-)$, because $Z$ is the canonical model of $K_X+\Delta$, by Theorem \ref{canonical bundle formula on fibration}, there is a generalised pair $(Z',B_{Z'}+\M_{Z'})$ and a birational morphism $ h:Z'\rightarrow Z$ such that $Z$ is the canonical model of $K_{Z'}+B_{Z'}+\M_{Z'}$. Let $B_Z$ be the pushforward of $B_{Z'}$, then $K_Z+B_Z+\M_Z$ is ample. 
		
		By Theorem \ref{thm: gen eff bir}, there is an integer $r'>0$ which only depends on $\Ii,l$ such that $|r'(K_{Z'}+B_{Z'}+\M_{Z'})|$ defines a birational map. Choose a general member $H'\in |r'(K_{Z'}+B_{Z'}+\M_{Z'})|$, let $H:=h_*H' $, then $H$ is ample and the coefficients of $H$ are bounded below by a positive number $\delta'$. Because by definition of the canonical model, $h^*H\leq H'$, by Theorem \ref{canonical bundle formula on fibration}(3) ,
		$$H^0(X,\mathcal{O}_X(ml(K_X+\Delta)))\cong H^0({Z'},\mathcal{O}_{Z'}(ml(K_{Z'}+B_{Z'}+\M_{Z'}))),$$ 
		and we have that
		\begin{equation}
			\Ivol(K_X+\Delta+f^*H)\leq (1+r)^d\Ivol(K_X+\Delta),
		\end{equation}
		then $(X,\Delta), H$ satisfies the conditions in Theorem \ref{boundedness of Ambro model}. Let $(\Cz,\pp)\rightarrow T$ be the bounded family and $\Cm_k,k\in\Lambda$ be the $\mathbb{Q}$-divisors defined in Theorem \ref{boundedness of Ambro model}, $\Cd'$ the set of generalised klt pairs $({W'},B_{W'}+\M_{W'})$ such that
		\begin{itemize}
			\item $({W'},B_{W'})$ is log smooth,
			\item There is a morphism $\phi:{W'}\rightarrow \Cz_t$ for a closed point $t\in T$,
			\item $\M_{W'}=\phi^*(\Cm_k|_{\Cz_t})$, for some $k\in\Lambda$, and
			\item $\mathrm{coeff}(B_{W'})$ is in a fixed DCC set and $\phi_*(B_{W'})\leq \pp_t$. 
		\end{itemize}
		Since $Z$ is the canonical model of $(X,\Delta)$, $\Ivol(X,K_X+\Delta)=\mathrm{vol}(Z,K_Z+B_Z+\M_Z)$. Let $({Z'}, B_{Z'}+\M_{Z'})$ be a generalized pair defined by Theorem \ref{canonical bundle formula on fibration} such that there is a birational morphism $\psi_t: {Z'}\rightarrow \Cz_t$ for a closed point $t\in T$. By Theorem \ref{canonical bundle formula on fibration}(3), $\mathrm{vol}(Z,K_Z+B_Z+\M_Z)=\mathrm{vol}({Z'},K_{Z'}+B_{Z'}+\M_{Z'})$, and by Theorem \ref{boundedness of Ambro model}, $\psi_{t*}B_{Z'}\leq \pp_t$ and $ \M_{Z'}=\psi_t ^*(\Cm_k)$, then $({Z'},B_{Z'}+\M_{Z'})\in \Cd'$, and
		$$\{\Ivol(K_X+\Delta)\ |\ (X,\Delta)\in \Cd(n,\Ii,l,r,C^-) \}\subset \{\mathrm{vol}(K_{Z'}+B_{Z'}+\M_{Z'})\ |\ ({Z'},B_{Z'}+\M_{Z'})\in \Cd' \}.$$ 
		By Theorem \ref{boundedness of Ambro model means DCC}, the later one is a DCC set, then $\{\Ivol(K_X+\Delta)\ |\ (X,\Delta)\in \Cd(n,\Ii,l,r,C^-) \}$ is a DCC set.
	\end{proof}

	\section{Boundedness of canonical models}
	In this section, we follow the method of \cite[Chapter 7]{HMX18}.
	\begin{defn}
		Let $(Z,B)$ be a log pair. Define a $\bm{b}$-divisor $\mathbf{M}_B$ by assigning to any valuation $\mu$,
		\begin{equation}
			\mathbf{M}_B(\mu)= \begin{cases}
				\mathrm{mult}_\Gamma(B), &\text{if the centre of }\mu\text{ is a divisor }\Gamma \text{ on }Z,\\
				1, &\text{otherwise.} 
			\end{cases}
		\end{equation}
	\end{defn}
	
	\begin{thm}\label{achieve volume}
		Fix a positive rational number $w$, a natural number $n$ and a DCC set $\Ii\subset [0,1]\cap \mathbb{Q}$. Let $(W, D+M)$ be a projective log smooth $n$-dimensional pair where $D$ is reduced, $M$ is nef, and $\mathbf{M}_D$ the $\bm{b}$-divisor defined as above. Then there exists $f:W'\rightarrow W$, a finite sequence of blow-up along strata of $\mathbf{M}_D$, such that if 
		\begin{itemize}
			\item $(Z,B)$ is a projective log smooth $n$-dimensional pair,
			\item $g:Z\rightarrow W$ is a finite sequence of blow-up along strata of $\mathbf{M}_D$,
			\item $\mathrm{coeff}(B)\subset \Ii$,
			\item $g_*B\leq D$, and
			\item $\mathrm{vol}(Z,K_Z+B+g^*M)=w$,
		\end{itemize}
		then $\mathrm{vol}(W',K_{W'}+\mathbf{M}_{B,W'}+f^*M)=w$, where $\mathbf{M}_{B,W'}$ is the strict transform of $B$ plus the $W'\dashrightarrow Z$ exceptional divisors.
		\begin{proof}
			We may assume that $1\in \Ii$. Let 
			$$V=\{\mathrm{vol}(Y,K_Y+\Gamma+g^*M)\ |\ (Y,\Gamma)\in \Cs \},$$
			where $\Cs$ is the set of all $n$-dimensional projective log smooth pairs $(Y,\Gamma)$ such that
			\begin{itemize}
				\item $K_Y+\Gamma+g^*M$ is big,
				\item $\mathrm{coeff}\Gamma\subset \Ii$, and
				\item there is a birational morphism $g:Y\rightarrow W$ and $g_*\Gamma \leq D$.
			\end{itemize}
			By Theorem \ref{DCC for fixed base}, $V$ satisfies the DCC condition, then there is a constant $\delta>0$ such that if $w\leq \mathrm{vol}(K_Y+\Gamma+g^*M)\leq w+\delta$, then $\mathrm{vol}(K_Y+\Gamma+g^*M)=w$. By Theorem \ref{pseudo eff threshold}, there exists $0<e<1$ such that if $(Y,\Gamma)\in \Cs$, then $K_Y+e\Gamma +g^*M$ is big. Fix $\epsilon >0$ such that 
			$$(1-\epsilon)^n>\frac{w}{w+\delta},\ \ \mathrm{and\ let\ } a=1-(1-e)\epsilon.$$
			Since
			$$K_Y+a\Gamma+g^*M =(1-\epsilon)(K_Y+\Gamma+g^*M)+\epsilon (K_Y+e\Gamma+g^*M),$$
			it follows that
			\begin{equation}
				\begin{aligned}
					\mathrm{vol}(K_Y+a\Gamma+g^*M) &\geq \mathrm{vol}((1-\epsilon)(K_Y+\Gamma+g^*M)\\
					&=(1-\epsilon)^n\mathrm{vol}(K_Y+\Gamma +g^*M)\\
					&\geq \frac{w}{w+\delta}\mathrm{vol}(K_Y+\Gamma +g^*M)
				\end{aligned}
			\end{equation}
			
			Since $(W,aD)$ is klt, there is a sequence of blow up $f:W'\rightarrow W$ of the strata of $\mathbf{M}_D$ with the following property: if $K_{W'}+\Phi=f^*(K_W+aD)+E$, where $\Phi \wedge E=0$, then $(W',\Phi)$ is terminal. Let $\Cs'$ be the set of pairs $(Z, B)$ satisfying the conditions given in this theorem, after blowing up $Z$ along finitely many strata of $\mathbf{M}_D$ and replacing $B$ by the image of $\mathbf{M}_{B}$, we may assume that $\phi:Z\rightarrow W'$ is a morphism. If $(Z,B)\in \Cs$ and $B_{W'}=\phi_*B$, then $af_*B_{W'}\leq aD$ so that if
			$$K_{W'}+\Phi=f^*(K_{W}+af_*B_{W'})+F$$
			where $\Phi \wedge F=0$, then $(W',\Phi)$ is terminal. We have
			\begin{equation}
				\begin{aligned}
					\mathrm{vol}(W',K_{W'}+aB_{W'}+f^*M) &= \mathrm{vol}(W',K_{W'}+aB_{W'}\wedge \Phi +f^*M)\\
					&= \mathrm{vol}(Z,K_Z+\phi_*^{-1}(aB_{W'}\wedge \Phi) +g^*M)\\
					&\leq \mathrm{vol}(Z,K_Z+B+g^*M),
				\end{aligned}
			\end{equation}
			and then,
			$$w\leq \mathrm{vol}(W',K_{W'}+B_{W'}+f^*M)\leq \frac{w+\delta}{w}\mathrm{vol}(W',K_{W'}+aB_{W'}+f^*M)\leq w+\delta,$$
			which implies $\mathrm{vol}(W',K_{W'}+B_{W'}+f^*M)=w$ as required.
		\end{proof}
	\end{thm}
	
	\begin{proof}[Proof of Theorem \ref{key theorem 2}]
		Define $\Cd(n,\Ii,l,r,C)$ to be 
		$$\{(X,\Delta)\ |\ (X,\Delta)\in \Cd(n,\Ii,l,r) \text{\ and\ }\ \Ivol(K_X+\Delta)=C \}.$$ 
		It suffices to show that there is an integer $N>0$ such that if $(Z, B_Z+\M_Z)$ is the canonical model of $(X,\Delta)\in \Cd(n,\Ii,l,r, C)$, then $N(K_Z+B_Z+\M_Z)$ is very ample. 
		Suppose this is not the case, let $\{(X_i,\Delta_i)\in \Cd(n,\Ii,l,r,C),i\in \mathbb{N}\}$ be a sequence, $(Z_i,B^i_{Z_i}+\M^i_{Z_i})$ the corresponding canonical model such that $i!(K_{Z_i}+B^i_{Z_i}+\M^i_{Z_i})$ is not very ample for all $i\in \mathbb{N}$.
		
		Let $({Z'_i},B^i_{{Z'_i}}+\M^i_{{Z'_i}}),i\in \mathbb{N}$ be the generalised pair defined in Theorem \ref{canonical bundle formula on fibration}(3), then $(Z_i,B^i_{Z_i}+\M^i_{Z_i})$ is the canonical model of $({Z'_i},B^i_{{Z'_i}}+\M^i_{{Z'_i}})$. By Theorem \ref{boundedness of Ambro model}, after passing to a stratification of $T$, we may assume that there is a log bounded log smooth family $(\Cz,\pp)\rightarrow T$ and finitely many $\mathbb{Q}$-divisors $\Cm_k,k\in \Lambda$ on $\Cz$, such that for every ${Z'_i}$, there is a closed point $t_i\in T$ and a birational morphism $\phi_i: {Z'_i}\rightarrow \Cz_{t_i}$, such that $\phi_{i*}B^i_{{Z'_i}}\leq \pp_{t_i}$, and passing to a subsequence, we may assume $\Lambda=\{1\}$, and then $\M^i_{{Z'_i}}=\phi_i ^* \Cm_1$.
		
		Applying Theorem \ref{achieve volume} to $(\Cz_{t_1},\pp_{t_1}+\Cm_1|_{\Cz_{t_1}})$, we obtain a model $\Cz'_{t_1}\rightarrow \Cz_{t_1}$ and $g:(\Cz',\pp':=\mathbf{M}_{\pp,\Cz'})\rightarrow \Cz$ the morphism obtained by blowing up the corresponding strata of $\mathbf{M}_\pp$.
		Define $\Phi_{t_i}=(f'_i)^{-1}_*B^i_{{Z'_i}}+\mathrm{Exc}(f'_i)\leq \pp'_{t_i}$, where $f'_i:\Cz'_{t_i}\dashrightarrow {Z'_i}$ is the induced birational map. Passing to a subsequence, we may also assume that for any irreducible component $P$ of the support of $\pp'$, the coefficients of $\Phi_{t_i}$ along $P_{t_i}$ are non-decreasing. Let $\Phi^i$ be the $\mathbb{Q}$-divisor supported on $\pp'$ such that $\Phi^i|_{\Cz'_{t_i}}=\Phi_{t_i}$, then the coefficients of $\Phi^i$ are non-decreasing.
		
		We claim that for any $i>0$, we have
		$$\mathrm{vol}(\Cz'_{t_i},K_{\Cz'_{t_i}}+\Phi_{t_i} +g^*\Cm_1|_{\Cz'_{t_i}})=C.$$
		To prove this, we may fix $i$. Applying Theorem \ref{achieve volume} to       $(\Cz_{t_i},\pp_{t_i}+\Cm_1|_{\Cz_{t_i}})$, we obtain a model               $\Cz''_{t_i}\rightarrow \Cz_{t_i}$ and $g:(\Cz'',\pp'':=\mathbf{M}_{\pp,\Cz''})\rightarrow \Cz$ the corresponding morphism obtained by blowing-up the corresponding strata of $\mathbf{M}_\pp$. By Theorem \ref{achieve volume} we have
		$$\mathrm{vol}(\Cz''_{t_i},K_{\Cz''_{t_i}}+\Psi_{t_i} +g'^*g^*\Cm_1|_{\Cz''_{t_i}})=C,$$
		where $\Psi_{t_i}:=\mathbf{M}_{B^i_{{Z'_i}},\Cz''_{t_i}}$. If $\Psi$ is the divisor supported on $\Supp(\mathbf{M}_{\pp',\Cz''})$ such that $\Psi|_{\Cz''_{t_i}}=\Psi_{t_i}$, then
		\begin{equation}
			\begin{aligned}
				C 	&=\mathrm{vol}(\Cz''_{t_i},K_{\Cz''_{t_i}}+\Psi_{t_i} +g'^*g^*\Cm_1|_{\Cz''_{t_i}})\\
				&=\mathrm{vol}(\Cz''_{t_1},K_{\Cz''_{t_1}}+\Psi|_{\Cz''_{t_1}} +g'^*g^*\Cm_1|_{\Cz''_{t_1}})\\
				&=\mathrm{vol}(\Cz'_{t_1},K_{\Cz'_{t_1}}+\Phi^i|_{\Cz'_{t_1}}+g^*\Cm_1|_{\Cz'_{t_1}})\\
				&=\mathrm{vol}(\Cz'_{t_i},K_{\Cz'_{t_i}}+\Phi_{t_i}+g^*\Cm_1|_{\Cz'_{t_i}}),
			\end{aligned}
		\end{equation}
		where the second and the fourth equalities follow from Theorem \ref{Invariance of Plurigenra} and the third one follows from Theorem \ref{achieve volume}.

		Because $(\Cz'_{t_1},\pp'_{t_1})$ is log smooth, for any common resolution $Y$ of $Z'_1$ and $\Cz'_{t_1}$, we have $\mathbf{M}_{\Phi_{t_1},Y}\geq \mathbf{M}_{B^1_{Z_1},Y}$. Also Because $ \mathrm{vol}(Z'_1,K_{Z'_1}+B'_{Z'_1}+\M'_{Z_1})=\mathrm{vol}(\Cz'_{t_1},K_{\Cz'_{t_1}}+\Phi_{t_1} +g^*\Cm_1|_{\Cz'_{t_1}})$, by Lemma \ref{log canonical model}, $(Z'_1,B'_{Z'_1}+\M'_{Z_1})$ has the same canonical model with $(\Cz'_{t_1},\Phi_{t_1} +g^*\Cm_1|_{\Cz'_{t_1}})$, which is $(Z_1,B^1_{Z_1}+\M^1_{Z_1})$, then there is a birational contraction $\Cz'_{t_1}\dashrightarrow Z_1$.
		
		Since $(Z_1, B^1_{Z_1}+\M^1_{Z_1})$ is generalized klt and the contraction $\Cz'_{t_1}\dashrightarrow Z_1$ is $(K_{\Cz'_{t_1}}+\Phi_{t_1} +g^*\Cm_1|_{\Cz'_{t_i}})$-non-negative, we may slightly increase the coefficients of the corresponding exceptional divisor in $\Phi_{t_1}$ to make the contraction $\Cz'_{t_1}\dashrightarrow Z_1$ to be $(K_{\Cz'_{t_1}}+\Phi_{t_1} +g^*\Cm_1|_{\Cz'_{t_i}})$-positive. Let $\mathcal{A}$ be an ample divisor on $\Cz'$, then we can choose $\epsilon \ll 1$ such that $\Cz'_{t_1}\dashrightarrow Z_1$ to be $(K_{\Cz'_{t_1}}+\Phi_{t_1} +(g^*\Cm_1+\epsilon \mathcal{A})|_{\Cz'_{t_i}})$-positive, also because $\epsilon$ is small enough,
		$Z_1$ is the canonical model of $(\Cz'_{t_1},\Phi_{t_1}+(g^*\Cm_1+\epsilon \mathcal{A})|_{\Cz'_{t_1}})$.
		
		Because $(g^*\Cm_1+\epsilon \mathcal{A})|_{\Cz'_{t_1}}$ is ample, there is an open subset $U\subset T$ such that $g^*\Cm_1+\epsilon \mathcal{A}$ is ample over $U$. To prove the boundedness we can replace $T$ with a stratification, so we can assume that $U=T$. By \cite[Corollary 1.4]{HMX18}, $(\Cz',\Phi^1+g^*\Cm_1+\epsilon \mathcal{A})$ has a relative log canonical model $\psi:\Cz'\dashrightarrow \Cw$ over $T$, which fiber by fiber $\psi_{t_i}:\Cz'_{t_i}\dashrightarrow \Cw_{t_i}$ gives the log canonical model for $(\Cz'_{t_i},\Phi^1|_{\Cz'_{t_i}}+(g^*\Cm_1+\epsilon \mathcal{A})|_{\Cz'_{t_i}})$ for all $i\geq 1$. Because by assumption, $\Cw_{t_1}$ is also the log canonical model of $(\Cz'_{t_1},\Phi^1|_{\Cz'_{t_1}}+g^*\Cm_1|_{\Cz'_{t_1}})$, then $K_{\Cw_{t_1}}+\psi_*(\Phi^1+g^*\Cm)|_{\Cw_{t_1}}$ is ample. After replacing $T$ by an open subset, we have $K_{\Cw}+\psi_*(\Phi^1+g^*\Cm)$ is ample over $T$, in particular, $K_{\Cw_{t_i}}+\psi_*(\Phi^1+g^*\Cm)|_{\Cw_{t_i}}$ is ample. Because $\mathcal{A}$ is ample and $\psi:\Cz'\dashrightarrow \Cw$ is $(K_{\Cz'}+\Phi^1+g^*\Cm_1+\epsilon \mathcal{A})$-non-negative, by negativity lemma, $\psi$ is also $(K_{\Cz'}+\Phi^1+g^*\Cm_1)$-non-negative, hence $\psi_{t_i}:\Cz'_{t_i}\dashrightarrow \Cw_{t_i}$ is $(K_{\Cz'_{t_i}}+\Phi^1|_{\Cz'_{t_i}}+g^*\Cm_1|_{\Cz'_{t_i}})$-non-negative. Also because $K_{\Cw_{t_i}}+\psi_*(\Phi^1+g^*\Cm)|_{\Cw_{t_i}}$ is ample, $\psi_{t_i}:\Cz'_{t_i}\dashrightarrow \Cw_{t_i}$ is the log canonical model of $(\Cz'_{t_i},\Phi^1|_{\Cz'_{t_i}}+g^*\Cm_1|_{\Cz'_{t_i}})$.

		Notice that by Theorem \ref{Invariance of Plurigenra},
		$$C=\mathrm{vol}(\Cz'_{t_k},K_{\Cz'_{t_k}}+\Phi^k|_{\Cz'_{t_k}}+g^*\Cm_1|_{\Cz'_{t_i}})=\mathrm{vol}(\Cz'_{t_1},K_{\Cz'_{t_1}}+\Phi^k|_{\Cz'_{t_1}}+g^*\Cm_1|_{\Cz'_{t_1}})$$
		for all $k>0$. Since we have assumed that 
		$$\Phi^1\leq \Phi^2\leq \Phi^3\leq ...,$$
		it follows by Lemma \ref{log canonical model} that $\psi_{t_i}:\Cz'_{t_i}\dashrightarrow \Cw_{t_i}$ is also a log canonical model of $(\Cz'_{t_i},\Phi^k|_{\Cz'_{t_i}}+g^*\Cm_1|_{\Cz'_{t_i}})$ for all $k\geq 1$, $\psi_{t_i*}\Phi ^k|_{\Cz'_{t_i}}=\psi_{t_i*}\Phi^1|_{\Cz'_{t_i}} $ and there is an isomorphism $\alpha_i: Z_i\rightarrow \Cw_{t_i}$.
		There is an integer $N >0$ such that $N(K_\Cw+\psi_* \Phi^1+\psi_* g^*\Cm_1)$ is very ample over $T$ and so $N(K_{\Cw_{t_i}}+\psi_* \Phi^1|_{\Cz'_{t_i}}+\psi_* g^*\Cm_1|_{\Cw_{t_i}})$ is very ample for all $i>0$. Since
		$$\psi_{t_i*} \Phi^1|_{\Cz'_{t_i}}=\psi_{t_i*}\Phi^i|_{\Cz'_{t_i}}=\alpha_{i*}B^i_{Z_i},$$
		we have $N(K_{Z_i}+B^i_{Z_i}+\M^i_{Z_i})$ is very ample for all $i>0$, which is the required contradiction.
	\end{proof}
	
	\begin{lemma}
		\label{log canonical model}
		Let $(W,B+M)$ be a generalized log canonical pair such that $K_W+B+M$ is big, $M$ is nef and $f:W\dashrightarrow Z$ the log canonical model of $(W,B+M)$. Suppose $B'\geq B$ is a $\mathbb{Q}$-divisor such that $(W,B'+M)$ is generalized log canonical and $\mathrm{vol}(W,K_W+B+M)=\mathrm{vol}(W,K_W+B'+M)$, then $f$ is also the log canonical model of $(W,B'+M)$
		\begin{proof}
			Replacing $W$ by a birational model, we may assume that $f:W\rightarrow Z$ is a morphism. Let $A=K_Z+f_*(B+M)$, then $A$ is ample. By the negativity lemma, $F:=K_W+B+M-f^*A$ is effective and has the same support as the $f$-exceptional divisor. We have for any $t\geq 0$
			\begin{equation}
				\begin{aligned}
					\mathrm{vol}(K_W+B+M) &=\mathrm{vol}(K_W+B+M+t(B'-B))\\
					&\geq \mathrm{vol}(f^*A+t(B'-B))\\
					&\geq \mathrm{vol}(f^*A)\\
					&=\mathrm{vol}(K_W+B+M).
				\end{aligned}
			\end{equation}
			Then 
			$$\mathrm{vol}(f^*A+t(B'-B))=\mathrm{vol}(f^*A),\ \ \ t\in [0,1]$$
			is a constant function. If $E$ is a component of $B'-B$ then by \cite{LM09} we have 
			\begin{equation}
				\begin{aligned}
					0	&=\frac{d}{dt}\mathrm{vol}(f^*A+tE)|_{t=0}\\
					&=n.\mathrm{vol}_{X|E}(f^*A)\\
					&\geq n.E.f^*A^{n-1}\\
					&=n.\mathrm{deg}f_*E.
				\end{aligned}
			\end{equation}
			Then $E$ is $f$-exceptional and so
			\begin{equation}
				\begin{aligned}
					H^0(W,\mathcal{O}_W(m(K_X+B'+M))) &= H^0(W,\mathcal{O}_W(mf^*A+m(E+F)))\\
					&=H^0(W,\mathcal{O}_W(mf^*A))\\
					&=H^0(W,\mathcal{O}_W(m(K_W+B+M)))
				\end{aligned}
			\end{equation}
			and thus $f$ is the log canonical model of $(W,B'+M)$.
		\end{proof}
	\end{lemma}
	\section{Weak boundedness}
	The definition of weak boundedness is introduced in \cite{KL10}. 
	\begin{defn}
		A $(g,m)$-curve is a smooth curve $C^o$ whose smooth compactification $C$ has genus $g$ and such that $C\setminus C^o$ consists of $m$ closed points.
	\end{defn}
	\begin{defn}
		Let $W$ be a proper $k$-scheme with a line bundle $\mathcal{N}$ and let $U$ be an open subset of a proper variety. We say a morphism $\xi:U\rightarrow W$ is weakly bounded with respect to $\mathcal{N}$ if there exists a function $b_{\mathcal{N}}:\mathbb{Z}^2_{\geq 0}\rightarrow \mathbb{Z}$ such that for every pair $(g,m)$ of non-negative integers, for every $(g,m)$-curve $C^o\subseteq C$, and for every morphism $C^o\rightarrow U$, one has that $\mathrm{deg}\xi_C ^*\mathcal{N}\leq b_{\mathcal{N}}(g,m)$, where $\xi_C:C\rightarrow W$ is the induced morphism. The function $b_{\mathcal{N}}$ will be called a weak bound, and we will say that $\xi$ is weakly bounded by $b_\mathcal{N}$.
		
		We say a quasi-projective variety $U$ is weakly bounded if there exists a compactification $i:U\hookrightarrow W$, such that $i:U\hookrightarrow W$ is weakly bounded with respect to an ample line bundle $\mathcal{N}$ on $W$. 
	\end{defn}
	\begin{lemma}\label{decompose to get weakly bounded}
		Let $T$ be a quasi-projective variety. Then we can decompose $T$ into finitely many locally closed subsets $\cup T_i$, such that each $T_i$ is weakly bounded.
		\begin{proof}
			By the definition of weakly bounded, if a variety $U$ is weakly bounded, then any open subset $U^o\subset U$ is also weakly bounded. Therefore, we may replace $T$ with a compactification and assume that $T$ is projective, and we only need to show that $T$ has a weakly bounded open subset. 
			
			Fix an integer $n\geq 2$, let $A$ be a general very ample divisor on $\mathbb{P}^n_T$ such that $K_{\mathbb{P}^n_T}+A$ is also very ample, then $\Supp A$ is smooth and dominates $T$, by the generic smoothness theorem, there is an open subset $T_1\subset T$ such that $\Supp A_{T_1}$ is smooth over $T_1$, where $A_{T_1}:=A|_{T_1}$. 
			
			Since $K_{\mathbb{P}^n_T}+A$ is ample and $A$ is smooth, by the adjunction formula, $K_{A}=(K_{\mathbb{P}^n_T}+A)|_{A}$ is very ample, then $A_{T_1}\rightarrow T_1$ is a family of canonically polarized manifolds. We may assume that $T_1$ is irreducible and every fiber of $A_{T_1}\rightarrow T_1$ has Hilbert polynomial $h(m)=\chi(\omega_{A_t}^{\otimes m})$.
			
			Write $\Cm^o _h$ for the (Deligne-Mumford) stack space of canonically polarized manifolds with Hilbert polynomial $h$ and $\mathbf{M}^o_h$ for its coarse moduli space. It is easy to see that $g:A_{T_1}\rightarrow T_1\in \Cm^o_h( T_1)$. Let $\psi: T_1\rightarrow \mathbf{M}^o_h$ be the induced moduli map. Because $A$ is very ample on $\mathbb{P}^n_T$ and $K_{\mathbb{P}^n_T}+A$ is very ample, $\psi$ is a finite morphism.
			
			By \cite[Lemma 6.2]{KL10}, the stack $\mathbf{M}^o_h$ is weakly bounded with respect to $\mathbf{M}_h$ and $\lambda\in \mathrm{Pic}(\mathbf{M}_h)$ by a function $b(g,d)$, where $\mathbf{M}_h$ is a compactification of $\mathbf{M}^o _h$, $\lambda $ is an ample line bundle. Let $T_1^c $ be a compactification of $T_1$ such that $\psi:T_1\rightarrow \mathbf{M}^o_ h$ extends to a morphism $\psi ^c:T_1 ^c \rightarrow \mathbf{M} _h$. Suppose $C^o\subseteq C$ is a $(g,d)$-curve, $C^o\rightarrow T_1$ be a morphism and $\xi:C\rightarrow T_1 ^c$ be its closure, then $\psi^c\circ \xi :C\rightarrow \mathbf{M}_h$ is the closure of $C^o\rightarrow T_1\xrightarrow{\psi }\mathbf{M}^o_h$. By the definition of weakly bounded, $ \mathrm{deg}(\psi^c\circ \xi)^* \lambda \leq b(g,d)$, then $ T_1$ is weakly bounded with respect to $T_1^c$ and $\psi^* \lambda$. 
			Because $\psi$ is a finite morphism, $ \psi^*\lambda$ is ample, then $T_1$ is weakly bounded.
			
		\end{proof}
	\end{lemma} 
	
	\begin{thm}{\cite[Proposition 2.14]{KL10}}\label{weakly bounded morphisms are bounded}
		Let $T$ be a quasicompact quasiseparated reduced $\mathbb{C}$-scheme and $\Cu\rightarrow T$ a smooth morphism. Given a projective $T$-variety and a polarization over $T$, $(\Cm,\mathcal{O}_{\Cm}(1))$, an open subvariety $\Cm^o\hookrightarrow \Cm$ over $T$, and a weak bound $b$, there exists a $T$-scheme of finite type $\Cw^b_{\Cm^o}$ and a morphism $\Theta:\Cw^b_{\Cm^o}\times \mathcal{U}\rightarrow \Cm^o$ such that for every geometric point $t\in T$ and for every morphism $\xi:\mathcal{U}_t\rightarrow \Cm^o_t\subset \Cm_t$ that is weakly bounded by $b$ there exists a point $p\in \Cw^b_{\Cm^o_t}$ such that $\xi=\Theta|_{\{p\}\times \mathcal{U}_t}$.
		
		In particular, if $\Cm^o$ is weakly bounded and $\Cm$ is the compactification, by definition, every morphism $\xi:\mathcal{U}_t\rightarrow \Cm^o_t\subset \Cm_t$ is weakly bounded by a weak bound $b$, then $\xi=\Theta|_{\{p\}\times \mathcal{U}_t}$ for a closed point $p\in \Cw^b_{\Cm^o_t}$.
	\end{thm}
	
	\section{Hilbert scheme and the moduli part}
	
	\subsection{Parametrizing space}

	A class of polarized Calabi-Yau pairs is a set $\mathscr{C}$ consisting of a triple pair $(X,\Delta,H)$ such that
	\begin{itemize}
		\item $H$ is an effective ample divisor,
		\item $K_X+\Delta\sim_{\mathbb{Q}}0$, and
		
		\item $(X,\Delta+\epsilon H)$ is log canonical for a positive number $\epsilon \ll 1$.
	\end{itemize} 	
	
	A family of polarized Calabi-Yau pairs over a normal base scheme $S$ consists of a flat, proper morphism $f: X\rightarrow S$, a $\mathbb{Q}$-divisor $\Delta$ on $X$ and a $\mathbb{Q}$-Cartier divisor $H$ such that $K_{X/S}+\Delta$ is $\mathbb{Q}$-Cartier and all fibers $(X_s,\Delta_s,H_s)$ are polarized Calabi-Yau pairs. We denote it by $(X,\Delta,H)\rightarrow S$.
	
	Given a class of polarized Calabi-Yau pairs $\mathscr{C}$, we define $\mathscr{MC}(S)$ to be the set of families of polarized Calabi-Yau pairs over $S$, $(X,\Delta, H)\rightarrow S$ such that $K_X+\Delta$ is $\mathbb{Q}$-Cartier and $(X_s,\Delta_s, H_s)\in\mathscr{C}$ for every closed point $s\in S$.

	Suppose $\mathscr{C}$ is a class of $n$-dimensional polarized Calabi-Yau pairs. We say $\mathscr{C}$ is bounded if the following two equivalent conditions hold:
	\begin{itemize}
		\item There exists a positive number $C$ and a positive integer $d$, such that for every $(Y,D,H)\in \mathscr{C}$, $dH$ is very ample without higher cohomology, $(dH)^{n}\leq C$, and $(dH)^{n-1}.\mathrm{red}(D)\leq C$.
		\item There is a flat projective morphism $\Cz\rightarrow S$ over a scheme of finite type, two divisors $\pp,\mathcal{L}$ on $\Cz$ which are flat over $S$ and a positive integer $d$, such that $S$ is of finite type, and for every $(Y,D,H)\in \mathscr{C}$, there is a closed point $s\in S$ and an isomorphism $\phi:(Y,dH)\rightarrow (\Cz_s,\mathcal{L}_s)$ such that $\phi_*D\leq \pp_s$. 
	\end{itemize} 
	
	If condition (1) holds, then there is a (non-unique) natural choice of the scheme $S$ in (2). By boundedness of Chow variety, we may assume that $Y$ has a fixed polynomial $H(t)$ with respect to $dH$, let $\mathbb{P}$ be the projective space of dimension $H(1)-1$ with a fixed coordinate system. By the proof of \cite[Proposition 5.17]{KP17}, because normal is an open condition, we may choose $\Ch'$ to be the locally closed subset of the Hilbert scheme of $\mathbb{P}$ which parametrizes all irreducible normal subvarieties of $\mathbb{P}$ with Hilbert polynomial $H(t)$ and $\mathscr{F}:\Cx_{\Ch'}\rightarrow \Ch'$ be the universal family. 
	
	Let $\Lambda$ be a finite set, $p_i(t),i\in \Lambda$ be $|\Lambda|$ polynomials such that $\mathrm{deg}p_i(t)=\mathrm{deg}H(t)-1$ for all $i$. Let
	$$\Ch_i:=\mathrm{Hilb}_{p_i(t)}(\Cx_{\Ch'} / \Ch') $$
	be the locally closed subset of the Hilbert scheme which parametrizes closed pure dimensional subschemes $D_i\subset \Cx_{\Ch'}$ such that 
	$D_i\rightarrow \Ch'$ is flat family of varieties with Hilbert polynomial $p_i(t)$. Let $\mathcal{D}_i\rightarrow \Ch_i$ be its universal family. For simplicity of notation, we define $\Ch:=\Ch_1\times_{\Ch'}...\times _{\Ch'}\Ch_{|\Lambda|}$ and 
	$$(\Cx_{\Ch},\mathcal{D}_{\Ch}):=(\Cx_{\Ch'}\times_{\Ch'} \Ch,\sum \mathcal{D}_i\times _{\Ch_i}\Ch).$$

	\begin{rem}\label{moduli map}
		Let $\mathscr{C}$ be a bounded class of polarised Calabi-Yau pairs. With the same notation as above, let $(X,\Delta)$ be a klt pair and $L$ is a divisor on $X$, suppose the general fiber of a contraction $f:(X,\Delta,L)\rightarrow Z$ is in $\mathscr{C}$, that is, there is an open subset $U\subset Z$ such that for every closed point $u\in U$, $(X_u,\Delta|_{X_u},L|_{X_u})\in \mathscr{C}$.
		
		Write $\Delta=\sum \Delta_i$ as the sum of irreducible components. Because the degree of $\Supp\Delta_i$ is bounded from above, by the boundedness of the Chow varieties, the Hilbert polynomial of $\Delta_{i,u}$ is in a finite set $p_i(t),i\in \Lambda$. Let $(\Cx_{\Ch},\mathcal{D}_{\Ch})\rightarrow \Ch$ be the family constructed as above. By the construction of $\Ch$, every closed point $u\rightarrow U$ corresponds to a closed point in $\Ch$ and there is a morphism $U\rightarrow \Ch$. 
		
		Notice that $\Delta_{i,u}$ may not be irreducible for every $u\in U$, two irreducible components of $\Delta_u$ may be considered as two divisors or just one divisor, depending on the divisor $\Delta_i$ on $X$. That means given two contractions $(X^i,\Delta^i)\rightarrow Z^i,i=1,2$ satisfying the given conditions, even if $(X^1_{u_1},\Delta^1_{u_1})\cong (X^2_{u_2},\Delta^2_{u_2})$, $u_1$ and $u_2$ may corresponds to different points in $\Ch$. 
		
		Since $dL$ is very ample without higher cohomology and $f:X\rightarrow Z$ is flat over $U$, $f_*\mathcal{O}_X(dL)$ is locally free over $U$. Replacing $U$ with an open subset, we may assume that $f_*\mathcal{O}_X(dL)$ is in fact free. Fixing a basis in the space of sections then gives a map $U\rightarrow \Ch'$ and $X_U\rightarrow U$ is isomorphic to the pull-back of the universal family $\Cx_{\Ch'}\rightarrow \Ch'$. Similarly, each irreducible component $\Delta_i$ of $\Delta$ gives a map $U\rightarrow \Ch_i$. Hence there is a morphism $\phi:U\rightarrow \Ch$ such that $f:(X_U,\Supp(\Delta_U))\rightarrow U$ is isomorphic to the pullback of $(\Cx_{\Ch},\mathcal{D}_{\Ch})\rightarrow \Ch$ by $\phi$.
		
		Suppose $\alpha=(\alpha_1,...,\alpha_k)$ is a vector of rational numbers and 
		$$\Delta_U=\alpha \Supp(\Delta_U):=\sum \alpha_i \Supp(\Delta_{i,U}),$$ 
		by the construction of $\mathcal{D}_\Ch$, $(X_U,\Delta_U)$ is isomorphic to the pullback of $(\Cx_{\Ch},\alpha \mathcal{D}_{\Ch})\rightarrow \Ch$ by $\phi$. If there is a point $u\in U$ such that $(X_u,\Delta_u)$ is a log Calabi-Yau pair, then $(\Cx_{\phi(u)},\alpha \mathcal{D}_{\phi(u)})$ is a log Calabi-Yau. If $\mathrm{coeff}(\Delta)\subset \Ii$ is a DCC set, then by \cite[Theorem 1.5]{HMX14}, $\alpha \mathcal{D}_{\Ch}$ is in a finite set and there are only finitely many $\alpha \mathcal{D}_{\Ch}$.
		
		Moreover, by Lemma 7.4 in the first arxiv version of \cite{Bir23}, after replacing $\Ch$ by a stratification of a locally closed subvariety, we may assume that $\Ch$ is smooth and $(\Cx_{\Ch},\alpha \mathcal{D}_{\Ch})$ is klt Calabi-Yau over $\Ch$, then $(\Cx_{\Ch},\alpha \mathcal{D}_{\Ch})\rightarrow\Ch$ is an lc-trivial fibration. 
	\end{rem}

	\subsection{Moduli part}
	In this section, we deal with algebraic fibrations whose general fibers are log Calabi-Yau pairs. We claim that such a contraction naturally induces an lc-trivial fibration, then any such fibration has a moduli $\bm{b}$-divisor by the canonical bundle formula.
	
	\begin{thm}
		\label{general fiber calabi-yau will induce lc-trivial fibration structure}
		
		Let $(X,\Delta)$ be an lc pair and $f:(X,\Delta)\rightarrow Z$ an algebraic contraction to a projective normal $\mathbb{Q}$-factorial variety, suppose that the general fiber $(X_g,\Delta_g)$ is a log Calabi-Yau pair. Assume that there is a crepant birational morphism $g:(X',\Delta')\rightarrow (X,\Delta)$ and a divisor $D$ on $Z$, such that the morphism $h:=f\circ g:X'\rightarrow Z$ is smooth over $Z\setminus D$ and $\Supp(\Delta')$ is simple normal crossing over $Z\setminus D$.
		
		Then there is a $\mathbb{Q}$-divisor $\Lambda'$ on $X'$, such that
		\begin{itemize}
			\item $(X'_\eta,\Lambda '_{X_\eta})\cong (X'_\eta,\Delta '_{X_\eta})$, where $\eta$ is the generic point of $Z$.
			\item $\Supp(\Lambda')$ is log smooth over $Z\setminus D$ and,
			\item $(X',\Lambda')\rightarrow Z$ is an lc-trivial fibration.
		\end{itemize}
		\begin{proof}
			Since $(X_g,\Delta_g)$ is a Calabi-Yau variety, we have $K_{X'_\eta}+\Delta_\eta' \sim_{\mathbb{Q}} 0$, then there exists a vertical $\mathbb{Q}$-divisor $B'$ such that $K_{X'}+\Delta'+B'\sim_{\mathbb{Q}} 0$.  
			
			Suppose $B'=R+G$, where $\Supp(R)\not\subset h^{-1}(\Supp(D))$ and $ \Supp(G)\subset h^{-1}(\Supp(D))$. Because $R$ is vertical, $h$ is smooth over the generic point of $h(\Supp R)$ and $Z$ is $\mathbb{Q}$-factorial, $h(R)$ is a well defined $\mathbb{Q}$-Cartier divisor on $Z$, denote it by $R_Z$. Also because $h$ is smooth over $Z\setminus \Supp(D)$, there exists a $\mathbb{Q}$-divisor $F_R$ supported on $h^{-1}(\Supp(D))$, such that $R+F_R=h^*R_Z$, then $K_{X'}+\Delta'+B'-(R+F_R)\sim_{\mathbb{Q},h} 0$. Let $\Lambda':=\Delta'+B'-(R+F_R)$, then $K_{X'}+ \Lambda'\sim_{\mathbb{Q}} 0$, and $\Lambda'_\eta= \Delta'_\eta$. Write $\Delta'=\Delta'_{\geq 0}-\Delta'_{\leq 0}$, because $\Delta'_{\leq 0}$ is $g$-exceptional, it is easy to see that $(X',\Lambda')\rightarrow Z$ is an lc-trivial fibration. Because $\Supp(\Delta')$ is log smooth over $Z\setminus D$, $\Supp(F_R)\subset h^{-1}(D)$ and $\Supp(B'-R)\subset h^{-1}(D)$, then $\Supp(\Lambda')$ is log smooth over $Z\setminus D$. 
		\end{proof}
		
	\end{thm}

	\begin{prop}
		\label{moduli part stable under base change}
		Let $f:(X,\Delta)\rightarrow Z$ be an lc-trivial vibration between projective normal varieties, $\rho: Z'\rightarrow Z$ a surjective morphism from a projective normal variety $Z'$ and $f':(X',\Delta')\rightarrow Z'$ an lc-trivial fibration induced by the normalization of the main component of the base change,
		$$\xymatrix{
			(X,\Delta)  \ar[d] _{f}&    (X',\Delta') \ar[l]_{\rho_X} \ar[d]_{f'}\\
			Z&    Z'\ar[l]_{\rho}	 
		} $$
		Let $\M$ and $\M'$ be the moduli $\mathbf{b}$-divisors of $f$ and $f'$. Then
		\begin{enumerate}
			\item[(1)] if $\M_Z$ descends on $Z$ and $\M'_{Z'}$ descends on 
			$Z'$, then $\rho^*\M_Z=\M'_{Z'}$, and
			\item[(2)] if $\rho$ is finite and $\M_Z$ is $\mathbb{Q}$-Cartier,   then $\rho^*\M_Z=\M'_{Z'}$. In particular, $\M$ descends on $Z$ if 
			and only if $\M'$ descends on $Z'$.
		\end{enumerate}
		\begin{proof}
			(1) is \cite[Proposition 3.1]{Amb05}.
			
			For (2), let $g':W'\rightarrow Z$ and $g:W\rightarrow Z$ be birational morphisms such that $\M'$ descends on $W'$, $\M$ descends on $W$ and $\rho:Z'\rightarrow Z$ lifts to a morphism $\rho_W:W'\rightarrow W$. Then $\rho_W^*\M_W =\M'_{W'}$ by (1). Because $\rho$ is finite, any $g$-exceptional divisor is not dominated by divisors on $W'$, then the push-forward of $\rho_W^*\M_W =\M'_{W'}$ to $Z'$ gives $\rho^*\M_Z=\M'_{Z'}$.
		\end{proof}
	\end{prop}
	\begin{thm}
		\label{restriction of moduli part is the moduli part of restriction}
		Let $(X,\Delta)$ be a log canonical pair, $f:(X,\Delta)\rightarrow Z$ an lc-trivial fibration to a smooth projective variety $Z$, $g:(X',\Delta')\rightarrow (X,\Delta)$ a crepant birational morphism which is also a log resolution of $(X,\Delta)$. Suppose $D\subset Z$ is a smooth divisor on $Z$ such that $h:=g\circ f$ is smooth over the generic point $\eta_D$ of $D$. Let $Y$ be the normalization of the irreducible component of $f^{-1}(D)$ that dominates $D$, $\Delta_Y$ the $\mathbb{Q}$-divisor on $Y$ such that $K_Y+\Delta_Y=(K_{X}+\Delta+f^*D)|_Y$. Let $\M_Z$ denote the moduli part of $(X,\Delta)\rightarrow Z$. Suppose there is a reduced divisor $B$ on $Z$ such that $B+D$ is a reduced simple normal crossing divisor, the morphism $h: X'\rightarrow Z$ and $\Delta', B$ satisfy the standard normal crossing assumptions. Then $(Y,\Delta_Y)\rightarrow D$ is an lc-trivial fibration and its moduli part $\M_D$ is equal to $\M_Z|_D$.
		\begin{proof}
			By assumption, $h$ is smooth over $Z\setminus B$ and $D$ is smooth, the singular locus of $h^{-1}(D)$ is contained in $h^{-1}(B)\cap h^{-1}(D)$. After blowing up a sequence of smooth subvarieties whose centers are contained in the singular locus of $h^{-1}(D)$, we may assume that $(X',\Delta'+h^*(B+D))$ is log smooth. It is easy to see that the morphism $h: X'\rightarrow Z$ and $\Delta', B$ also satisfy the standard normal crossing assumption. 
			
			Let $E'$ be the irreducible component of $h^*D$ that dominates $D$, $\Delta'_{E'}$ the $\mathbb{Q}$-divisor on $E'$ such that $K_{E'}+\Delta'_{E'}=(K_{X'}+\Delta'+h^*D)|_{E'}$.  It is easy to see that the generic fiber of $(E',\Delta'_{E'})\rightarrow D$ is equal to the generic fiber of $(Y,\Delta_Y)\rightarrow D$, which means they have the same moduli part, then we only need to prove the result for $(E',\Delta'_{E'})\rightarrow Z$.

			By the canonical bundle formula, there is a divisor $B_Z$ supported on $B$ such that
			\begin{equation}
				\label{equation 3}
				K_{X'}+\Delta'+h^*D\sim_{\mathbb{Q}} h^*(K_Z+B_Z+\M_Z+D),
			\end{equation}
			and
			\begin{equation}
				\label{equation 4}
				K_X+\Delta+f^*D\sim_{\mathbb{Q}} f^*(K_Z+B_Z+\M_Z+D).
			\end{equation}
			
			Because $(Z,B+D)$ is log smooth, it is an lc pair. By Theorem \ref{canonical bundle formula} (c), $\Delta_v'+h^*D+h^*(B-B_Z)\leq \mathrm{red}(h^*(B+D))$. Since $(X',\Delta'+h^*(B+D))$ is log smooth, $(X',\Delta'+h^*D+h^*(B-B_Z))$ is sub-lc. 
			Because $Z$ is smooth, $B-B_Z$ is $\mathbb{Q}$-Cartier, after replacing $\Delta'$ by $\Delta'+h^*(B-B_Z)$, $\Delta$ by $\Delta+f^*(B-B_Z)$ and $B_Z$ by $B_Z+(B-B_Z)=B$, we can assume that $B=B_Z$ and every irreducible component of $B$ is dominated by an irreducible component of $\Delta'$ which has coefficient 1. Since $K_{X'}+\Delta'+h^*D\sim_{\mathbb{Q}}g^*( K_X+\Delta+f^*D)$, $(X,\Delta+f^*D)$ is log canonical.
			
			Let $g(E')=E$, suppose $h^*D=E'+E'_1,f^*D=E+E_1$. Restricting (\ref{equation 3}) to $E'$ and (\ref{equation 4}) to $E$, by the adjunction formula, there is a $\mathbb{Q}$-divisor $ \Delta'_{E'}$ and an effective $\mathbb{Q}$-divisor $\Delta_{E^n}$, such that
			$$(K_{X'}+\Delta'+h^*D)|_{E'}\sim_{\mathbb{Q}}K_{E'}+\Delta'_{E'}\sim_{\mathbb{Q}} h_{E'}^*(K_D+B|_D+\M_Z|_D),$$
			$$(K_X+\Delta+f^*D)|_{E^n}\sim_{\mathbb{Q}}K_{E^n}+\Delta_{E^n}\sim_{\mathbb{Q}} f_{E}^*(K_D+B|_D+\M_Z|_D),$$
			where $E^n$ is the normalization of $E$. It is easy to see that $ \Delta'_{E'}=\Delta'|_{E'} +E'_1|_{E'}$, $(E',\Delta'_{E'})$ is sub-lc, $\Delta_{E^n}$ is effective and $K_{E'}+\Delta'_{E'}\sim_{\mathbb{Q}} g_{E'}^*(K_{E^n}+\Delta_{E^n})$, where $g_{E'}:E'\rightarrow E^n$ is the birational morphism induced by $g|_{E'}:E'\rightarrow E$. It is easy to see $\Delta'_{E',\leq 0}$ is $g_{E'}$-exceptional, then $(E',\Delta'_{E'})\rightarrow D$ is an lc-trivial fibration.
			
			$$\xymatrix{
				(E',\Delta'_{E'}) \ar@{^{(}->}[r] \ar[d]^{g'_{E'}}&    (X',\Delta') \ar[d]_{g} \ar@/^2pc/[dd]^{h}\\
				(E^n,\Delta_{E^n}) \ar@{^{(}->}[r] \ar[d]^{f_E}  &   (X,\Delta)  \ar[d]_{f}\\
				D \ar@{^{(}->}[r]&    Z	 
			} $$
			
			By the canonical bundle formula for $(E',\Delta'_{E'})\rightarrow D$, we have
			\begin{equation}
				K_{E'}+\Delta'_{E'}\sim_{\mathbb{Q}} h_{E'} ^*(K_D+B_D+\M_D).
			\end{equation}
			To prove $\M_D=\M_Z|_D$, we only need to prove that $B_D\sim_{\mathbb{Q}} B|_D$.
			
			Since the morphism $X'\rightarrow Z$ and $\Delta',B$ satisfy the normal crossing assumption, $\M_Z$ descends on $Z$. Similarly, because $B+D$ is snc, $(D,B|_D)$ is log smooth and $(E',\Delta'_{E'})$ is log smooth over $D\setminus B\cap D$, $\M_D$ descends on $D$. For the same reason, the morphism $E'\rightarrow D$ and $\Delta'_{E'}, B|_D$ satisfy the standard normal crossing assumption. By the construction of the boundary divisor, $B_D$ is the unique smallest $\mathbb{Q}$-divisor supported on $B|_D$ such that
			$$\Delta'_{E',v} +h_{E'}^*(B|_D-B_D)\leq \mathrm{red}(h_{E'}^*(B|_D)),$$
			where $\Delta'_{E',v}$ is the vertical part of $\Delta'_{E'}$.
			Because every irreducible component of $B$ is dominated by an irreducible component of $\Delta'$ which has coefficient 1, every irreducible component of $B|_D$ is dominated by an irreducible component of $\Delta'_{E'}=\Delta'|_{E'} +E'_1|_{E'}$ which has coefficient 1, then $B|_D=B_D$ and the result follows.
		\end{proof}
	\end{thm}
	
	\begin{thm}[{\cite[Theorem 3.3]{Amb05}}]
		\label{moduli part is nef and good}
		Let $f:(X,\Delta)\rightarrow S$ be an lc-trivial fibration such that the generic fiber $X_{\eta}$ is projective variety and $\Delta_{\bar{\eta}}$ is effective, then there exists a diagram
		$$\xymatrix{
			(X,\Delta)  \ar[d] _{f}& &   (X^!,\Delta^!) \ar[d]_{f^!} &\\
			S 	\ar@/_1pc/[rrr] _\Phi &  \bar{S}\ar[l]_{\tau} \ar[r]_{\rho}& 	 S^! \ar@/_1pc/@{-->}[ll]_i  \ar[r]^{\pi} & S^*
		} $$
		satisfying the following properties:
		\begin{enumerate}
			\item[(1)] $f^!:(X^!,\Delta^!)\rightarrow S^!$ is an lc-trivial fibration.
			\item[(2)] $\tau$ and $\pi$ are generically finite and surjective morphisms, $\rho$ is surjective.
			\item[(3)] There exists a nonempty open subset $U\subset \bar{S}$ and an isomorphism
			$$\xymatrix{
				(X,\Delta)\times_S \bar{S}|_U  \ar[rr] _{\cong}\ar[dr]& &   (X^!,\Delta^!)\times_{S^!} \bar{S}|_U \ar[dl]\\
				&  U& 	 
			} $$
			\item[(4)] Let $\M$, $\M^!$ be the corresponding moduli-$\bm{b}$-divisors, write $\rho:=\Phi\circ \tau$, then $\M^!$ is $\bm{b}$-nef and big. And if $\M$ descends on $S$ and $\M^!$ descends on $S^!$, then $\tau^*\M_S=\rho^*\M_{S^!}^!$. 
			\item[(5)] There is a morphism $\Phi:S\rightarrow S^*$, which is an extension of the period map defined in \cite[Chapter 2]{Amb05}, and a rational map $i:S^!\dashrightarrow S$.
		\end{enumerate}
		
		Although it is not written in \cite{Amb05}, (5) is implied by its proof.
	\end{thm}

	\begin{thm}\label{bound singular locus}
		Let $(\Cx_{\Ch},\alpha\Cd_{\Ch})\rightarrow \Ch$ be the lc-trivial fibration defined in Remark \ref{moduli map}. Because $\alpha$ is effective, by Theorem \ref{moduli part is nef and good}, we have the following diagram    
		$$\xymatrix{
			(\Cx_{\Ch},\alpha\Cd_{\Ch})\ar[d]^{\Cf} &          &(\Cx^!_{\Ch^!},\alpha\Cd^!_{\Ch^!}) \ar[d]^{\Cf^!}& \\
			\Ch \ar@/_1pc/[rrr] _\Phi &\bar{\Ch} \ar[l]^\tau \ar[r]_\rho  &\Ch^!  \ar[r]^\pi  \ar@/_1pc/@{-->}[ll]_i    & \Ch^*
		}$$
		Furthermore, by Lemma \ref{decompose to get weakly bounded}, we can replace $\Ch^*$ by a stratification and $\Ch,\bar{\Ch},\Ch^!$ by its preimage, such that 
		\begin{itemize}
			\item $\tau$ and $\pi$ are finite,
			\item $\Ch^*$ is weakly bounded and smooth, and
			\item $(\Cx_{\Ch},\Cd_{\Ch})\rightarrow \Ch$ and $(\Cx^!_{\Ch^!},\Cd^!_{\Ch^!})\rightarrow \Ch^!$ have fiberwise log resolutions.
		\end{itemize} 
		Then there exists a positive integer $l$, such that
		if $f:(X,\Delta)\rightarrow Z$ is an lc trivial fibration such that
		\begin{itemize}
			\item there is a rational map $\phi:Z\dashrightarrow \Ch$, and
			\item the generic fiber of $f$ is isomorphic to the generic fiber of the pullback of $(\Cx_{\Ch},\alpha \Cd_{\Ch})\rightarrow \Ch$ by $\phi$. 
		\end{itemize}
		Then we can choose the moduli part $\M_{Z}$ of $f$, such that $\M_{Z}$ is effective, $l\M_{Z}$ is $\mathbf{b}$-Cartier, and if $\Ch^*\hookrightarrow \Cs^*$ is a smooth compactification such that $\Phi\circ \phi$ extends to a morphism $Z\rightarrow \Cs^*$, then $\Supp(\M_{Z})\supset Z\setminus U$, where $U=\Phi\circ \phi^{-1}\Ch^*$.
		\begin{proof}
			Let $(\Cy_\Ch,\Cr_\Ch)\rightarrow (\Cx_{\Ch},\Cd_{\Ch}),(\Cy^!_{\Ch^!},\Cr^!_{\Ch^!})\rightarrow (\Cx^!_{\Ch^!},\Cd^!_{\Ch^!})$ be crepant birational morphisms which are fiberwise log resolution of $\Cf,\Cf^!$. After taking smooth compactifications of the bases $\Ch,\bar{\Ch},\Ch^!$ and $\Ch^*$, and choosing extensions of the fibrations, we have the following diagram,
			$$\xymatrix{
				(\Cy_{\Cs},\Cr_{\Cs})\ar[d]^{\Cf} &          &(\Cy^!_{\Cs^!},\Cr^!_{\Cs^!}) \ar[d]^{\Cf^!}& \\
				\Cs \ar@/_1pc/[rrr] _\Phi &\bar{\Cs} \ar[l]^\tau \ar[r]_\rho  &\Cs^!  \ar[r]^\pi  \ar@/_1pc/@{-->}[ll]_i    & \Cs^*
			}$$
			Furthermore, by choosing the compactification properly, we may assume that the moduli part $\bm{\Cm}^!_{\Cs^!}$ of $\mathcal{F}^!$ descends on $\Cs^!$ and the moduli part $\bm{\Cm}_{\Cs}$ of $\mathcal{F}$ descends on $\Cs$. By Theorem \ref{moduli part is nef and good}, we have $ \tau^* \bm{\Cm}_{\Cs}=\rho^*{\bm{\Cm}^!_{\Cs^!}}$. 
			
			Let $\mathcal{N}$ be a Cartier divisor on $\Cs^*$ such that $\Ch^*\subset \Cs^*\setminus \Supp \mathcal{N}$. Because $\bm{\Cm}^!_{\Cs^!}$ is big, we may fix a section of $|\bm{\Cm}^!_{\Cs^!}|_{\mathbb{Q}}$ such that $\Supp \bm{\Cm}^!_{\Cs^!} \supset \Supp \pi^* \mathcal{N}$, then $\Ch^!\subset \Cs^! \setminus \Supp \bm{\Cm}^!_{\Cs^!}$. Because $ \tau^* \bm{\Cm}_{\Cs}=\rho^*{\bm{\Cm}^!_{\Cs^!}}$, we can choose $\bm{\Cm}_{\Cs}$ such that $\tau(\Supp \rho^* \bm{\Cm}^!_{\Cs^!})\subset \Supp \bm{\Cm}_{\Cs}$.
			
			Let $h:Z'\rightarrow Z$ be a birational map, by the negativity lemma, $\M_{Z'}\leq f^* \M_Z$, then $\Supp \M_{Z'}\subset f^{-1}\Supp\M_Z$. To prove $\Supp (\M_Z)\supset Z\setminus U$, we only need to proved that $\Supp(\M_{Z'})\supset Z'\setminus h^{-1}(U)$, so we can replace $Z$ by a higher birational model such that $\M_Z$ descends on $Z$.
			
			Define $\mathcal{B}:=\Cs\setminus \Ch$. We have the following two cases:
			
			Case 1: if the generic point of $\phi(Z)$ is contained in $\Supp(\bm{\Cm}_{\Cs}+\mathcal{B})$. After passing to a stratification of $\Cs$, we may replace $(\Cy,\Cr)\rightarrow \Cs$ by its restriction to irreducible components of $\Supp(\bm{\Cm}_{\Cs}+\mathcal{B})$ and repeat this process. Since the dimension strictly decreases, this will stop.
			
			Case 2: if the generic point of $\phi(Z)$ is not contained in $\Supp(\bm{\Cm}_{\Cs}+\mathcal{B})$. We claim that $\phi^*\bm{\Cm}_{\Cs}\sim_{\mathbb{Q}}\M_Z$. Because 
			$\Supp \bm{\Cm}_{\Cs} \supset \tau(\Supp \rho^* \bm{\Cm}^!_{\Cs^!}) \supset \Supp \Phi^* \mathcal{N}$, $\Ch^*\subset \Cs^* \setminus \Supp \mathcal{N}$ and $\phi^{-1}(\Ch)=U$, we have $\Supp \M_Z \supset Z\setminus U$. Since $l\bm{\Cm}_{\Cs}$ is Cartier, $l\M_Z$ is also Cartier.
		\end{proof}
	\end{thm}
	\begin{proof}[Proof of claim]
		With the same notation as above. Let $\Cs_Z$ be the image of $Z$ in $\Cs$. Consider the following diagram
		$$\xymatrix@R=0.8em@C=1.5em{
			(\Cy_D,\Cr_D) \ar[dd] \ar[rr]	\ar[dr]	&			&	(\tilde{\Cy},\tilde{\Cr})	\ar[dd]\ar[dr]	&             \\
			&	(\Cy_{\Cs_Z},\Cr_{\Cs_Z})\ar[dd]\ar[rr]		&			&        (\Cy,\Cr)\ar[dd]     \\
			D	\ar@{^{(}->}[rr] _j \ar[dr]^g	&			&		\tilde{\Cs}\ar[dr]  ^h	&             \\
			&		{\Cs_Z}\ar@{^{(}->}[rr]	_j&			&      \Cs       \\
		} $$
		where 
		\begin{itemize}
			\item $D$ is a divisor on $\tilde{\Cs}$ dominates ${\Cs_Z}$,
			\item $(\tilde{\Cs},D+\mathcal{B}_{\tilde{\Cs}})$ is log smooth, where $\mathcal{B}_{\tilde{\Cs}}:=h^*\mathcal{B}$,
			\item $\tilde{\Cs}\rightarrow \Cs$ is a birational morphism, and
			\item $(\Cy_{\Cs_Z},\Cr_{\Cs_Z}) \rightarrow \Cs_Z$, $(\Cy_D,\Cr_D)\rightarrow D$ and $(\tilde{\Cy},\tilde{\Cr})\rightarrow \tilde{\Cs}$ are induced by pullback of $(\Cy,\Cr)\rightarrow \Cs$.
		\end{itemize}
		Then $h(\tilde{\Cs}\setminus \mathcal{B}_{\tilde{\Cs}})=\Ch$ and $(\tilde{\Cy},\Supp\tilde{\Cr})\rightarrow \tilde{\Cs}$ is log smooth over $\tilde{\Cs}\setminus  \mathcal{B}_{\tilde{\Cs}}$. 
		
		Because the generic fiber of $(X,\Delta)\rightarrow Z$ is crepant birationally equivalent to the generic fiber of the pullback of $(\Cy,\Cr)\rightarrow \Cs$ via $\phi$, it is also crepant birationally equivalent to the generic fiber of the pullback of $(\Cy_{\Cs_Z},\Cr_{\Cs_Z})\rightarrow \Cs_Z$, then by Theorem \ref{moduli part stable under base change}, the moduli part $\M_{\Cs_Z}$ of $(\Cy_{\Cs_Z},\Cr_{\Cs_Z})\rightarrow \Cs_Z$ satisfies $\M_Z=\phi^* \M_{\Cs_Z}$.
		
		Because the generic point of $\phi(Z)$ is in $\Ch$, then $D\not\subset \Supp(\mathcal{B}_{\tilde{\Cs}}) $, by Theorem \ref{restriction of moduli part is the moduli part of restriction}, the induced morphism $(\Cy_D,\Cr_D)\rightarrow D$ is an lc-trivial fibration, and the corresponding moduli divisor $\M_D$ is equal to $\bm{\Cm}_{\tilde{\Cs}}|_D$. By Theorem \ref{moduli part stable under base change}, $\M_D=g^*\M_{\Cs_Z}$, then $\M_Z=\phi ^*\M_{\Cs_Z}= \phi^*\bm{\Cm_S}$.
	\end{proof}
	
	\begin{thm}\label{getting bounded Ambro models}
		We use the same notation as in Theorem \ref{bound singular locus}.
		Suppose there is a family of smooth varieties $\mathcal{U}\rightarrow T$ and a morphism $\Theta:\mathcal{U} \rightarrow \Ch^*$, define $\bar{\Cu}:=\Cu\times_{\Ch^*}\Ch^!$.
		
		Then after passing to a stratification of $T$, there is a compactification of $\mathcal{U}\hookrightarrow \Cz/T$, such that, for any closed point $t\in T$, if $(X,\Delta)\rightarrow Z$ is an lc-trivial fibration, such that
		\begin{itemize}
			\item there is a birational morphism $Z\rightarrow \Cz_t$, and
			\item there exists a finite cover $V\rightarrow \bar{\Cu}_t$ such that the generic fiber of $(X,\Delta)\times_{\Cz_t}V\rightarrow V$ is crepant birationally equivalent to the generic fiber of $(\Cy^!_{\Ch^!},\Cr^!_{\Ch^!})\times_{\Ch^!}V\rightarrow V$,
		\end{itemize}
		then the moduli part of $(X,\Delta)\rightarrow Z$ descends on $\Cz_t$.
		\begin{proof}
			After replacing $\Cu,\Ch^*,\Ch^!$ by open subsets, we may assume that $\Ch^!\rightarrow \Ch^*$ is \'etale and $\bar{\Cu}\rightarrow T$ is a family of smooth varieties. Then $\bar{\Cu}\rightarrow \Cu$ is \'etale, we may let $K(\tilde{\Cu})/K(\Cu)$ be the Galois closure of $K(\bar{\Cu})/K(\Cu)$, and $\tilde{\Cu}\rightarrow \Cu$ be the Galois cover with group $G$. After replacing $\Cu$ by an open subset and passing $T$ to a stratification, we may assume that $\tilde{\Cu}_t\rightarrow \Cu_t$ is \'etale between smooth varieties for every closed point $t\in T$.
			
			Suppose $\tilde{\Cu}\hookrightarrow \tilde{\Cz}'$ is an extension, let $\tilde{\Cz}\rightarrow \tilde{\Cz}'$ be a $G$-equivariant log resolution of $(\tilde{\Cz}',\tilde{\Cz}'\setminus \tilde{\Cu})$. Let $\Cz$ be the quotient of $\tilde{\Cz}$ by $G$. Next, we show that $\Cz$ satisfies the requirements.
			
			Write $(\Cy^!_{\tilde{\Cu}_t},\Cr^!_{\tilde{\Cu}_t}):=(\Cy^!_{\Ch^!},\Cr^!_{\Ch^!})\times _{\Ch^!}\tilde{\Cu}_t$, it is easy to see $(\Cy^!_{\tilde{\Cu}_t},\Supp\Cr^!_{\tilde{\Cu}_t})$ is log smooth over $ \tilde{\Cu}_t$. Because $(\tilde{\Cz}_t,\tilde{\Cz}_t\setminus \tilde{\Cu}_t)$ is log smooth, by \ref{standard normal crossing}, the moduli part of $(\Cy^!_{\tilde{\Cu}_t},\Cr^!_{\tilde{\Cu}_t})\rightarrow \tilde{\Cu}_t$ descends on $\tilde{\Cz}_t$, denote it by $\tilde{\bm{M}}_{\tilde{\Cz}_t}$. 
			
			Suppose $(X,\Delta)\rightarrow Z$ is an lc-trivial fibration that satisfies the conditions, denote its moduli part by $\bm{M}_{\Cz_t}$.
			Because $V\rightarrow \bar{\Cu}_t\rightarrow \Cu_t$ is finite, we can choose a compactification $V\hookrightarrow W$ such that there is a finite morphism $ W\rightarrow \Tilde{\Cz_t}$. By assumption, the generic fiber of $(X,\Delta)\times_{\Cz_t} V\rightarrow V$ is crepant birationally equivalent to the generic fiber of $(\Cy^!_{\Ch^!},\Cr^!_{\Ch^!})\times_{\Ch^!}V\rightarrow V$.
			Because the moduli part only depends on the crepant birational equivalent class of generic fiber, by Proposition \ref{moduli part stable under base change}, $\M$ descends on $\Cz_t$ if and only if the moduli part of $(X,\Delta)\times_{\Cz_t} V\rightarrow V$ descends on $W$ if and only if the moduli part of $(\Cy^!_{\tilde{\Cu}_t},\Cr^!_{\tilde{\Cu}_t})\rightarrow \tilde{\Cu}_t$ descends on $\tilde{\Cz}_t$.
		\end{proof}
	\end{thm}

	\section{Proof of the Main Theorems}
	
	\begin{proof}[Proof of Theorem \ref{fibers in a good moduli space}]
		We use the same notation as in Remark \ref{moduli map} and Theorem \ref{bound singular locus}.
		
		Let $C>0$ be a fixed number, to prove the DCC=, we only need to prove that if $\Ivol(K_X+\Delta)\leq C$, then $\Ivol(K_X+\Delta)$ is in a DCC set. By Theorem \ref{canonical bundle formula on fibration}, we can construct a generalized pair $(Z',B_{Z'}+\M_{Z'})$ and birational morphism $Z'\rightarrow Z$, such that $\mathrm{coeff}(B_{Z'})$ belongs to a DCC set $\Lambda$, $\Ivol(K_X+\Delta)=\mathrm{vol}(K_{Z'}+B_{Z'}+\M_{Z'})$ and $\M_{Z'}$ is the moduli part of $f$. After replacing $Z$ by $Z'$, $B_{Z'}$ and $\M_{Z'}$ by $B_{Z}$ and $\M_{Z}$, we only need to prove that $\mathrm{vol}(K_{Z}+B_{Z}+\M_{Z})$ belongs to a DCC set. 
		
		By assumption, since the general fiber of $(X_g,\Delta_g,L_g)$ is in $\mathscr{C}$, there is an open subset $U\hookrightarrow Z$ such that $(X_U,\Delta|_{X_U})$ is crepant birationally equivalent to the pullback of $(\Cy_{\Ch},\Cr_{\Ch})\rightarrow \Ch$ by a morphism $U\rightarrow \Ch$. After resolving the indeterminacy of $Z\dashrightarrow \Cs\dashrightarrow \Cs^*$, we may assume that there is a morphism $\phi: Z\rightarrow \Cs^*$. 
		
		By the proof of Theorem \ref{boundedness of Ambro model}, after replacing $Z$ by a birational model, there is a birational contraction $g:Z\rightarrow W$ and a very ample divisor $A$ on $W$ such that
		\begin{enumerate}
			\item $W$ is smooth
			\item $g^*A+F\sim r(K_{Z}+B_{Z}+(2d+1)l\M_{Z})$ for an effective $\mathbb{Q}$-divisor $F\geq 0$.
		\end{enumerate}
		We increase $l$ by 1 and assume that $\Supp(\M_{Z})\subset \Supp(F)$.
		
		Next, we construct a birational open subset of $Z$ which maps into $\Ch^*$ via $\phi$ and is in a bounded family.
		
		Let $Z\dashrightarrow Z_c$ be the canonical model of $K_{Z}+B_{Z}+(2d+1)l\M_{Z}+(2d+1)\phi^*H+(2d+1)g^*A$. By \cite[Lemma 4.4]{BZ16}, $Z\dashrightarrow Z_c$ is $\M_{Z}$, $g^*A$ and $\phi^*H$-trivial. Then there are two morphisms, $g':Z_c\rightarrow W$ and $\phi':Z_c\rightarrow \Cs^*$. Let $B_{Z_c}, \M_{Z_c}$ and $F_c$ be the push-forward of $B_{Z}, \M_{Z}$ and $F$ on $Z_c$, then $K_{Z_c}+B_{Z_c}+(2d+1)l\M_{Z_c}+(2d+1)\phi'^*H+(2d+1)g'^*A$ is ample, $l\M_{Z_c}$ is nef and Cartier. Because $K_{Z_c}+B_{Z_c}+(2d+1)l\M_{Z_c} \sim_{\mathbb{Q}} \frac{1}{r}(g'^*A+F_c)$, then $\frac{1}{r}(g'^*A+F_c) +(2d+1)\phi'^*H+(2d+1)g'^*A$ is ample, denote it by $\mathcal{A}$, clearly $\mathcal{A}$ is effective.
		$$\xymatrix{
			& Z \ar[dl]_{g} \ar[dr]^{\phi}  \ar@{-->}[d]&    \\
			W& Z_c\ar[r] ^{\phi'} \ar[l]_{g'} & 	 \Cs^*
		} $$
		Because $\mathrm{coeff}(B_{Z})$ is in a DCC set $\Lambda$, $Z$ is smooth and $r(K_{Z}+B_{Z}+(2d+1)l\M_{Z}) \sim g^*A+F,$ $\{r(K_{Z}+B_{Z}+(2d+1)l\M_{Z}) \}=\{ rB_{Z}\}=\{F \}$, then $\mathrm{coeff}(F)$ is in a DCC set $\Lambda'=\Lambda'(\Lambda,r)$, in particular, there is a positive number $\delta$ such that $\mathrm{coeff}(F)>\delta$.
		
		\textbf{Claim}. $(W,\Supp(g'_*\mathcal{A}))$, which is equal to $ (W,\Supp(A+ g'_*(\phi'^*H+F_c)))$, is log bounded.

		Because $\mathcal{A}$ is ample, $W$ is smooth, then $g'(\Supp\mathcal{A})$ is pure of codimension 1, and $g'(\Supp\mathcal{A})=\Supp(g'_*\mathcal{A})$. By the negativity lemma, $\mathcal{A}=g'^*g'_*\mathcal{A}-\mathcal{E}$, where $\mathcal{E}$ is an effective exceptional $\mathbb{Q}$-divisor such that $\Supp(\mathcal{E})=\mathrm{Exc}(g')$. Because $\mathcal{A}\geq 0$, we have that $\mathrm{Exc}(g')\subset \Supp(g'^*g'_*\mathcal{A})$ and
		$$W\setminus \Supp(g'_*\mathcal{A})\cong Z_c\setminus \Supp(g'^*g'_*\mathcal{A}).$$
		
		By Theorem \ref{bound singular locus}, $\phi(Z\setminus \Supp \M_Z)\subset \Ch^*$. Because $ \Supp(\M_{Z})\subset \Supp(F)$ and $\phi(Z\setminus \Supp(\M_{Z}))\subset \Ch^*$, then $\Supp(\M_{Z_c})\subset \Supp(F_c)\subset \Supp(\mathcal{A})$ and $\phi'(Z_c \setminus \Supp(\M_{Z_c}))\subset \Ch^*$. Let 
		$$U_c:=Z_c\setminus \Supp(g'^*g'_*\mathcal{A})=W\setminus \Supp(g'_*\mathcal{A}),$$ 
		It is easy to see that $U_c\subset Z_c\setminus \Supp \M_{Z_c}$ and $\phi'(U_c)\subset \Ch^*$.
		
		Because $(W,g'_*\mathcal{A})$ is log bounded, there is a family of variety $\Cu\rightarrow T$ over a scheme of finite type $T$ and a closed point $t\in T$ such that 
		$$\Cu_t\cong W\setminus g'_*\mathcal{A}\cong U_c.$$
		Because $\Ch^*$ is weakly bounded, by Theorem \ref{weakly bounded morphisms are bounded}, there exists a finite type scheme $\mathscr{W}$ and a morphism $\Theta:\mathscr{W}\times \Cu\rightarrow \Ch^*$ such that $\phi'=\Theta|_{\{p\}\times \Cu_t}$ for a closed point $p\in\mathscr{W}$. We replace $\Cu\rightarrow T$ by $\mathscr{W}\times \Cu\rightarrow \mathscr{W}\times T$. 
		
		Let $V:=U\times_{\Ch}\bar{\Ch}$, then $V\rightarrow U$ is a finite cover. By Theorem \ref{moduli part is nef and good}, the generic fiber of $(X,\Delta)\times_Z V\rightarrow V$ is crepant birationally equivalent to the generic fiber of $(\Cy^!_{\Ch^!},\Cr^!_{\Ch^!})\times_{\Ch^!} V\rightarrow V$. Then by Theorem \ref{getting bounded Ambro models}, there is a compactification $\Cu\hookrightarrow \Cz /T$ such that the moduli part of $(X,\Delta)$ descends on $\Cz_t$. Therefore, by the proof of Theorem \ref{key theorem} and Theorem \ref{key theorem 2}, conclusions (i) and (ii) follow.
	\end{proof}
	\begin{proof}[Proof of claim]
		With the same notation as in the proof of Theorem \ref{fibers in a good moduli space}. Because $A$ and $H$ are integral divisors, $\mathrm{coeff}(F_c)$ is bounded from above and $A$ is very ample on $W$, by the boundedness of the Chow variety, we only need to prove that the following three intersection numbers, $A^d$, $A^{d-1}. g'_*\phi'^*H$ and $A^{d-1}. g'_*F_c$ are bounded from above.
		
		First we prove that there is a constant $C_1$ such that $A^d\leq C_1$. By Theorem \ref{pseudo eff threshold}, there is a rational number $e\in (0,1)$ such that $K_{Z}+B_{Z}+e\M_{Z}$ is big. By the log-concavity of the volume function, we have that
		\begin{equation}
			\begin{aligned}
				\mathrm{vol}(K_{Z}+B_{Z}+\M_{Z}) &\geq \lambda ^d\mathrm{vol}(K_{Z}+B_{Z}+e\M_{Z})\\
				&+(1-\lambda)^d\mathrm{vol}(K_{Z}+B_{Z}+(2d+1)l\M_{Z}),
			\end{aligned}
		\end{equation}
		where $\lambda =\frac{(2d+1)l-1}{(2d+1)l-e}<1$. By assumption, $\mathrm{vol}(K_{Z}+B_{Z}+\M_{Z})\leq C$,  then $\mathrm{vol}(K_{Z}+B_{Z}+(2d+1)l\M_{Z})\leq \frac{C}{(1-\lambda)^d}$ and $A^d\leq \mathrm{vol}(Z,r(K_{Z}+B_{Z}+(2d+1)l\M_{Z}))\leq \frac{r^dC}{(1-\lambda)^d}$.
		
		Secondly, we prove that $A^{d-1}. g'_*\phi'^*H$ is bounded from above, which is equivalent to proving that $A^{d-1}. g_*\phi^*H$ is bounded from above. The idea is to show that $A^{d-1}. g_*\phi^*H$ is equal to the degree of a line bundle on a $(g,m)$-curve, with $g+m$ bounded, then apply weak boundedness. 
		
		Let $A_1,...,A_{d-1}\in |g^*A|$ be $d-1$ general member of the linear system, because $g^*A$ is base point free, the $\Supp(A_i)$ are smooth and intersect along a smooth curve $C$. By the adjunction formula, 
		$$(g^*A)^{d-1}.(K_{Z}+B_{Z}+(2d+1)l\M_{Z}+(d-1)g^*A)=\mathrm{deg}(K_C+B_Z|_C+(2d+1)l\M_Z|_C).$$
		
		Consider the following diagram
		$$\xymatrix{
			& Z \ar[dl]_{g}   \ar@{-->}[d]&  \tilde{Z} \ar[l]_{h} \ar[ld]^{h_1}   \\
			W& Z_1 \ar[l]^{g_1} & 	 
		} $$
		where $Z_1$ is the canonical model of $K_{Z}+B_{Z}+(2d+1)l\M_{Z}+(2d+1)g^*A$ and $\tilde{Z}$ is a common resolution of $Z\dashrightarrow Z_1$. By \cite[Lemma 4.4]{BZ16}, $Z\dashrightarrow Z_1$ is $g^*A$-trivial, so there is a birational morphism $g_1:Z_1\rightarrow W$. By the projection formula,
		\begin{equation}\nonumber
			\begin{aligned}
				& (g^*A)^{d-1}.(K_{Z}+B_{Z}+(2d+1)l\M_{Z}+(2d+1)g^*A) \\
				=& (h^* g^*A)^{d-1}.(h^*(K_{Z}+B_{Z}+(2d+1)l\M_{Z}+(2d+1)g^*A)) \\
				=& (g_1 ^*A)^{d-1}.(h_{1*}h^*(K_{Z}+B_{Z}+(2d+1)l\M_{Z}+(2d+1)g^*A))\\
				=&(g_1 ^*A)^{d-1}.(K_{Z_1}+B_{Z_1}+(2d+1)l\M_{Z_1}+(2d+1)g_1^*A),
			\end{aligned}
		\end{equation}
		where $B_{Z_1}$ and $\M_{Z_1}$ are the push forwards of $B_{Z}$ and $\M_{Z}$. Since $Z_1$ is the canonical model of $K_{Z}+B_{Z}+(2d+1)l\M_{Z}+(2d+1)g^*A$,  $K_{Z_1}+B_{Z_1}+(2d+1)l\M_{Z_1}+(2d+1)g_1^*A$ is ample.
		Because $g_1^*A$ and $K_{Z_1}+B_{Z_1}+(2d+1)l\M_{Z_1}+(2d+1)g_1^*A$ are both nef, it is easy to see that
		\begin{equation}\nonumber
			\begin{aligned}
				&(g_1 ^*A)^{d-1}.(K_{Z_1}+B_{Z_1}+(2d+1)l\M_{Z_1}+(2d+1)g_1^*A)\\
				\leq& (K_{Z_1}+B_{Z_1}+(2d+1)l\M_{Z_1}+(2d+1)g_1^*A+g_1^*A)^d\\
				=&\mathrm{vol}(K_{Z_1}+B_{Z_1}+(2d+1)l\M_{Z_1}+(2d+2)g_1^*A).
			\end{aligned}
		\end{equation}
		Since $Z\dashrightarrow Z_1$ is $g^*A$-trivial, $Z_1$ is also the canonical model of $K_Z+B_Z+(2d+1)l\M_Z+(2d+2)g^*A$, then
		\begin{equation}\nonumber
			\begin{aligned}
				& \mathrm{vol}(K_{Z_1}+B_{Z_1}+(2d+1)l\M_{Z_1}+(2d+2)g_1^*A) \\
				=& \mathrm{vol}(K_{Z}+B_{Z}+(2d+1)l\M_{Z}+(2d+2)g^*A) \\
				\leq& \mathrm{vol}(K_{Z}+B_{Z}+(2d+1)l\M_{Z}+(2d+2)(g^*A+F))\\
				=&\mathrm{vol}((1+(2d+2)r)(K_{Z}+B_{Z}+(2d+1)l\M_{Z}))\\
				\leq& (\frac{1+(2d+2)r}{r})^dC_1.
			\end{aligned}
		\end{equation}
		Then we have 
		\begin{equation}
			\label{equation 1}
			\begin{aligned}
				&\mathrm{deg}(K_C+B_Z|_C+(2d+1)l\M_Z|_C)\\
				\leq&(g^*A)^{d-1}.(K_{Z}+B_{Z}+(2d+1)l\M_{Z}+(d-1)g^*A)\\
				\leq &(g^*A)^{d-1}.(K_{Z}+B_{Z}+(2d+1)l\M_{Z}+(2d+1)g^*A)\\
				\leq & (\frac{1+(2d+2)r}{r})^dC_1.
			\end{aligned}
		\end{equation}

		By the construction of $\M_{Z}$, $Z\setminus \Supp(\M_{Z})$ maps into $\Ch^*$, so $C\setminus \Supp(\M_{Z}|_C)$ maps into $\Ch^*$. Suppose $C^o:=C\setminus \Supp(\M_{Z}|_C)$ is a $(g,m)$-curve, then $m\leq \mathrm{deg}_C(l\M_Z|_C)$, and $2g-2+(2d+1)m\leq \mathrm{deg}(K_C+B_Z|_C+(2d+1)l\M_Z|_C)$ is bounded. Because $\Ch^*$ is weakly bounded with respect to $H$, and $C^o$ is a $(g,m)$-curve with $2g+(2d+1)m$ is bounded, then $(g^*A)^{d-1}.\phi^*H=C.\phi^*H=\mathrm{deg}_C(\phi^*H|_C)$ is bounded, by the projection formula, $A^{d-1}.g_*\phi^*H$ is bounded. 
		
		Thirdly, we prove that $A^{d-1}. g'_*F_c$ is bounded from above, which is equal to prove that $A^{d-1}. g_*F$ is bounded. Because $K_Z+B_Z+(2d+1)l\M_Z\sim_{\mathbb{Q}} \frac{1}{r}(g^*A+F)$, then 
		$$A^{d-1}. g_*F=(g^*A)^{d-1} .F\leq (g^*A)^{d-1} .r(K_Z+B_Z+(2d+1)l\M_Z)$$
		which is also bounded by (\ref{equation 1}). 
		
		Finally, because $(g^*A)^{d-1}.F $ and $(g^*A)^{d-1}.(\phi^*H)$ are both bounded and $\mathrm{coeff}(F)\geq \frac{1}{\delta}$, by the boundedness of the Chow variety, $(W,\Supp(g_*(g^*A+F +(2d+1)\phi^*H)))$ is log bounded.
	\end{proof}
	\noindent\textbf{Acknowledgement}.
	I would like to thank my advisor, Professor Christopher Hacon, for many useful suggestions, discussions, and his generosity. I would also like to thank Stefano Filipazzi, Zhan Li, Yupeng Wang, Jingjun Han, and Jihao Liu for many helpful conversations.


\nocite{*}
	
	\begin{thebibliography}{PTW02}
		
		\bibitem[AK00]{AK00} D. Abramovich and K. Karu, \textit{Weak semistable reduction in characteristic 0}, Invent. Math. \textbf{139} (2000), no. 2, 241--273.
		
		
		
		\bibitem[Amb05]{Amb05} F. Ambro, \textit{The moduli b-divisor of an lc-trivial fibration}, Compos. Math. \textbf{141} (2005), no. 2, 385--403.
		
		
		\bibitem[Bir12]{Bir12} C. Birkar, \textit{Existence of log canonical flips and a special LMMP}, Pub. Math. IHES., \textbf{115} (2012), 325--368.
		
		
		\bibitem[Bir18]{Bir18} C. Birkar, \textit{Log Calabi-Yau fibrations}, arXiv: 1811.10709v2.
		
		\bibitem[Bir19]{Bir19} C. Birkar, \textit{Anti-pluricanonical systems on Fano varieties}. Ann. of Math. (2), \textbf{190} (2019), 345--463.
		
		
		
		
		
		\bibitem[Bir21a]{Bir21a} C. Birkar, \textit{Singularities of linear systems and boundedness of Fano varieties}, Ann. of Math. \textbf{193} (2021), no. 2, 347--405.
		
		
		\bibitem[Bir21b]{Bir21b} C. Birkar, \textit{Boundedness and volume of generalised pairs}, arXiv: 2103.14935v2.
		
		\bibitem[Bir23]{Bir23} C. Birkar, \textit{Geometry and moduli of polarised varieties}, Pub. Math. IHES., (2023).
		
		\bibitem[BCHM10]{BCHM10} C. Birkar, P. Cascini, C. D. Hacon and J. M\textsuperscript{c}Kernan, \textit{Existence of minimal models for varieties of log general type}, J. Amer. Math. Soc. \textbf{23} (2010), no. 2, 405--468.
		
		\bibitem[BDCS20]{BDCS20} C. Birkar, G. Di Cerbo, and R. Svaldi, \textit{Boundedness of elliptic Calabi-Yau varieties with a rational section}, arXiv: 2010.09769v1.
		
		
		\bibitem[BZ16]{BZ16} C. Birkar and D.-Q. Zhang, \textit{Effectivity of Iitaka fibrations and pluricanonical systems of polarized pairs}, Pub. Math. IHES., \textbf{123} (2016), 283--331.
		
		
		
		
		
		
		\bibitem[Fil18]{Fil18} S. Filipazzi, \textit{Boundedness of Log Canonical Surface Generalized Polarized Pairs}, Taiwanese J. Math. \textbf{22} (2018), no.4, 813--850.
		
		
		
		\bibitem[Fil20]{Fil20} S. Filipazzi, \textit{On the boundedness of $n$-folds with $\kappa(X)=n-1$}, arXiv: 2005.05508v2. 
		
		
		
		
		
		
		
		
		
		
		
		
		
		
		
		
		\bibitem[FM00]{FM00} O. Fujino and S. Mori, \textit{A canonical bundle formula}, J. Differential Geom. \textbf{56} (2000), no. 1, 167--188.
		
		
		
		\bibitem[Gon11]{Gon11} Y. Gongyo, \textit{On the minimal model theory for dlt pairs of numerical Kodaira dimension zero}, Math. Rest. Lett. \textbf{18} (2011), no. 5, 991--1000.
		
		\bibitem[Gro66]{Gro66} A. Grothendieck, \textit{\'El\'ements de g\'eom\'etrie alg\'ebrique : IV. \'Etude locale des sch\'emas et des morphismes de sch\'emas, Seconde partie}, Publ. Math. IHES \textbf{28} (1966).
		
		
		
		\bibitem[HMX13]{HMX13} C. D. Hacon, J. M\textsuperscript{c}Kernan, and C. Xu, \textit{On the birational automorphisms of varieties of general type}, Ann. of Math. (2) \textbf{177} (2013), no. 3, 1077--1111.
		
		\bibitem[HMX14]{HMX14} C. D. Hacon, J. M\textsuperscript{c}Kernan, and C. Xu, \textit{ACC for log canonical thresholds}, Ann. of Math. \textbf{180} (2014), no. 2, 523--571.
		
		\bibitem[HMX18]{HMX18} C. D. Hacon, J. M\textsuperscript{c}Kernan, and C. Xu, \textit{Boundedness of moduli of varieties of general type}, J. Eur. Math. \textbf{20} (2018), no. 4, 865–901.
		
		
		
		\bibitem[HX13]{HX13} C. D. Hacon and C. Xu, \textit{Existence of log canonical closures}, Invent. Math. \textbf{192} (2013), no. 1, 161--195.
		
		
		
		
		
		
		
		
		
		
		
		
		
		
		
		
		
		\bibitem[Jia21]{Jia21} J. Jiao, \textit{On the finiteness of ample models}, Math. Rest. Lett. \textbf{29} (2022), no. 3, 763--784.
		
		
		
		\bibitem[Kaw98]{Kaw98} Y. Kawamata, \textit{Subadjunction of log canonical divisors, II}, Amer. J. Math. \textbf{120} (1998), 893--899.
		
		
		
		
		\bibitem[Kol07]{Kol07} J. Koll\'ar, \textit{“Kodaira’s canonical bundle formula and adjunction}. In: \textit{Flips for 3-folds and 4-folds}. Ed. by A. Corti. Vol. 35. Oxford Lecture Series in Mathematics and its Applications. Oxford: Oxford University Press, 2007. Chap. 8, 134--162.
		
		\bibitem[Kol13]{Kol13} J. Koll\'ar, \textit{Singularities of the minimal model program}, Cambridge Tracts in Math. \textbf{200} (2013),  Cambridge Univ. Press.
		
		
		\bibitem[KM92]{KM92} J. Koll\'{a}r and S. Mori, \textit{Classification of Three-Dimensional Flips}, Journal of the American Mathematical Society, vol. 5, no. 3, American Mathematical Society, 1992, pp. 533–703
		
		\bibitem[KM98]{KM98} J. Koll\'{a}r and S. Mori, \textit{Birational geometry of algebraic varieties}, Cambridge Tracts in Math. \textbf{134} (1998), Cambridge Univ. Press.
		
		\bibitem[KL10]{KL10} S. J. Kov\'acs and M. Lieblich, \textit{Boundedness of families of canonically polarized manifolds: A higher dimensional analogue of Shafarevich’s conjecture}, Ann. of Math. \textbf{172} (2010), no. 3, 1719--1748.
		
		\bibitem[KP17]{KP17} S. J. Kov\'acs and Z. Patakfalvi, \textit{Projectivity of the moduli space of stable log varieties and subadditivity of log-Kodaira dimension}, J. Amer. Math. Soc. \textbf{30} (2017), 959-1021.
		
		
		\bibitem[Laz04]{Laz04} R. Lazarsfeld, \textit{Positivity in algebraic geometry. II. Positivity for vector bundles and multiplier ideals.} Ergebnisse der Mathematik und ihrer Grenzgebiete. \textbf{3}. Folge.
		
		
		
		
		
		
		\bibitem[Li20]{Li20} Z. Li, \textit{Boundedness of the base varieties of certain fibrations}, arXiv: 2002.06565v2.
		
		
		
		\bibitem[LM09]{LM09} R. Lazarsfeld and M. Mustață, \textit{Convex bodies associated to linear series}, Ann. Sci. \'Ec. Norm. Sup\'er. (4) \textbf{42} (2009), no. 5, 783-835.
		
		
		
		
		
		
		
		
	\end{thebibliography}
\end{document}